\documentclass[final,1p,times]{elsarticle}

\usepackage{amssymb}
 \usepackage{amsthm}

 \usepackage{lineno}

\usepackage{amsmath}

\journal{arXiv}

\usepackage[utf8]{inputenc}
\usepackage{csquotes}
\usepackage{microtype}
\usepackage{mathrsfs}
\usepackage{amsmath}
\usepackage{amsfonts}
\usepackage{amssymb}
\usepackage{amsthm}
\usepackage{graphicx}
\usepackage{caption}
\usepackage{subcaption}
\usepackage{commath}
\usepackage[english]{babel}
\usepackage{xcolor}
\usepackage{todonotes}

\usepackage{subcaption}
\usepackage{enumitem}
\usepackage{listingsutf8}
\usepackage{lmodern}
\usepackage{etex}
\usepackage{expl3}
\usepackage{xparse}
\usepackage{verbatim}
\usepackage{multirow}

\usepackage[hypertexnames=false]{hyperref}

\makeatletter
\renewcommand*{\@pnumwidth}{3em}
\makeatother

\DeclareMathOperator{\ran}{ran}

\newcommand{\sgn}{\mathrm{sgn}} 
\newcommand{\gl}{\mathrm{GL}} 
\newcommand{\C}{\mathbb{C}}  
\newcommand{\R}{\mathbb{R}}
\newcommand{\Z}{\mathbb{Z}}
\newcommand{\N}{\mathbb{N}}

\renewcommand{\Im}{\operatorname{Im}}
\renewcommand{\Re}{\operatorname{Re}}

\newcommand{\F}{\mathcal{F}}

\newcommand{\id}{\mathrm{Id}}
\newcommand{\rb}[1]{\left(#1\right)}

\newcommand{\cb}[1]{\left\{#1\right\}}

\newcommand{\Hi}{\mathcal{H}}
\newcommand{\scp}[1]{\left\langle#1\right\rangle}
\newcommand{\gdw}{\Leftrightarrow}

\newcommand{\Span}{\mathrm{span}}

\newcommand{\ltwoC}[1][]{\ell^2\left(\Z^{#1};\C\right)}
\newcommand{\ltwoR}[1][]{\ell^2\left(\Z^{#1};\R\right)}

\newcommand{\CLone}[1][]{L^1\left(\R^{#1};\C\right)}
\newcommand{\RLtwo}[1][]{L^2\left(\R^{#1};\R\right)}
\newcommand{\CLtwo}[1][]{L^2\left(\R^{#1};\C\right)}

\newcommand{\fr}{\mathrm{Fr}}
\newcommand{\sur}{\mathrm{Sur}}

\newtheoremstyle{break}
{9pt}
{9pt}
{\itshape}
{}
{\bfseries}
{.}
{\newline}
{}

\theoremstyle{break}
\newtheorem{defn}{Definition}[section]
\newtheorem{theorem}[defn]{Theorem}
\newtheorem{lemma}[defn]{Lemma}
\newtheorem{corollary}[defn]{Corollary}
\newtheorem{remark}[defn]{Remark}
\newtheorem{example}[defn]{Example}
\newtheorem{proposition}[defn]{Proposition}
\allowdisplaybreaks 

\newlist{thmlist}{enumerate}{1}
\setlist[thmlist]{label=(\roman{thmlisti}), ref=\thethm.(\roman{thmlisti}),noitemsep}

\newcommand{\FH}[1]{\textcolor{blue}{{#1}}}
\begin{document}

\begin{frontmatter}



\title{Rebricking frames and bases\footnote{Version of June 26, 2023}}


\author[1]{Thomas Fink}
\ead{fink\_thomas@gmx.de}
\author[1]{Brigitte Forster\corref{cor1}}
\ead{brigitte.forster@uni-passau.de}
\author[1]{Florian Heinrich}
\ead{florian.heinrich@uni-passau.de}
\cortext[cor1]{Corresponding author}

\address[1]{Fakult\"at f\"ur Informatik und Mathematik, Universit\"at Passau, Germany}
%

\begin{abstract}
In 1949, Denis Gabor introduced the ``complex signal'' (nowadays called ``analytic signal'') by combining a real function $f$ with its Hilbert transform $Hf$ to a complex function $f+ iHf$. His aim was to extract phase information, an idea that has inspired techniques as the monogenic signal and the complex dual tree wavelet transform. In this manuscript, we consider two questions: When do two real-valued bases or frames $\{f_{n} : n\in\N\}$ and $\{g_{n} : n\in\N\}$ form a complex basis or frame of the form $\{f_{n} + i g_{n}: n\in\N\}$? And for which bounded linear operators $A$ forms  $\{f_{n} + i A f_{n} : n\in\N\}$ a complex-valued  orthonormal basis, Riesz basis or frame, when $\{f_{n} : n\in\N\}$ is a real-valued orthonormal basis, Riesz basis or frame? 
We call this approach \emph{rebricking}. It is well-known that the analytic signals don't span the complex vector space $L^{2}(\R; \C)$, hence $H$ is not a rebricking operator. We give a full characterization of rebricking operators for bases, in particular orthonormal and Riesz bases, Parseval frames, and frames in general.
We also examine the special case of finite dimensional vector spaces and show that we can use any real, invertible matrix for rebricking if we allow for permutations in the imaginary part.  
\end{abstract}



\end{frontmatter}



\section{Introduction}
\label{sec Introduction}

Let $\Hi$ always denote a real separable Hilbert space. Then $\Hi + i\Hi$ is a complex Hilbert space. In this article, we ask two questions:

Question 1:  When do two real-valued bases or frames $\{f_{n} : n\in\N\}$ and $\{g_{n} : n\in\N\}$ in $\Hi$ form a complex basis or frame of the form $\{f_{n} + i g_{n}: n\in\N\}$ in $\Hi + i\Hi$?

As the $f_{n}$ usually are not linear functionals, the classical and unique way of complexifying real-valued linear forms $u$ on some complex Banach space $X$ via
$$
U(x) = u(x) -i u(ix) \quad \mbox{for all }x\in X
$$
is not applicable for arbitrary functions, see \cite[Ch. 3]{Rudin1991}. Therefore, we are interested in necessary and/or sufficient conditions on the families $\{f_{n} : n\in\N\}$ and $\{g_{n} : n\in\N\}$.

It is well-known that Riesz bases are equivalent, hence for every pair of Riesz bases $\{f_{n} : n\in\N\}$ and $\{g_{n} : n\in\N\}$ in a Hilbert space $\Hi$ there exists a bounded invertible linear operator $A:\Hi \to \Hi$ such that $g_{n} = A f_{n}$ for all $n\in\N$.

From this point of view our second question is related:

Question 2: For which real-valued bounded linear operators $A: \Hi \to \Hi$ is  $$\{f_{n} + i A f_{n} : n\in\N\}$$ a complex-valued  orthonormal basis, Riesz basis or frame, when $\{f_{n} : n\in\N\}$ is a real-valued orthonormal basis, Riesz basis or frame? 
We call this approach \emph{rebricking}. We are interested in necessary and sufficient properties on the operator $A$.

Similar questions have been asked in signal processing.
For example, already in 1946, Denis Gabor introduced the \emph{complex signal}, nowadays called \emph{analytic signal}, for one-dimensional real-valued functions $f$ \cite{gabor1946}: His idea was to consider $$f_{a} = f + iH(f),$$ where $H$ denotes the Hilbert transform. From a frequency perspective, the operation is ``Suppress the amplitudes belonging to negative frequencies, and multiply the amplitudes of positive frequencies by two'', as Gabor states. The complexification $f + iH(f)$ of real-valued functions $f$ gives access to their amplitude and phase information, as the analytic signal has a meaningful polar representation $f_{a} = \psi_{a} \exp(i \phi_{a})$: The instantaneous amplitude  $\psi_{a}$ and the  instantaneous frequency $\phi_{a}$ give interpretable insight into the properties of the signals over time.

A more recent approach of complexification is the Dual Tree Complex Wavelet Transform by Nick Kingsbury \cite{Kingsbury98,Kingsbury99,Selesnick05}. The motivation is to produce a nearly shift-invariant wavelet transform, while still maintaining a much lower redundancy than the undecimated discrete wavelet transform. The ideal complex wavelet $\psi_{c}$ is of the form
$$
\psi_{c}(t) = \psi_{r}(t) + i \psi_{i}(t), \quad t\in\R,
$$
where---in analogy to the analytic signal---$\psi_{r}$is a real-valued wavelet, and $\psi_{i} = H(\psi_{r})$ is the Hilbert transform of $\psi_{r}$. However, such an ideal complex wavelet cannot be achieved exactly, when the wavelets are required to have compact support. Hence usually the complex wavelet transform operates with approximately analytic wavelets \cite{Selesnick05}.

Our article is structured as follows: First we consider the simplest non-redundant space-spanning systems, i.e., Riesz bases and---as a special case---orthonormal bases. Here we also discuss the case in finite-dimensional vector spaces. In the next section, we analyze the rebricking of frames, in particular Parseval frames.

We collect some notation: With the inner product $\langle \bullet, \bullet \rangle_{\Hi}$ for the real Hilbert-space $\Hi$, the inner product of $\Hi + i\Hi$ is defined by
 \[\langle f, g\rangle = \langle f_r + if_i,g_r+ig_i\rangle := \langle f_r,g_r \rangle_\Hi + \langle f_i,g_i\rangle_\Hi + i\left(\langle f_i,g_r\rangle_\Hi - \langle f_r , g_i \rangle_\Hi\right)\] for $f,g \in \Hi + i\Hi$ where $f_r,g_r \in \Hi$ are the real and $f_i,g_i \in \Hi$ are the imaginary parts. It is easy to see that this definition fulfils the property of a complex inner product. 
For the norm induced by the inner product it is clear that \begin{align*}\norm{f}^2 &= \langle f_r + if_i,f_r + if_i\rangle = \langle f_r,f_r \rangle_\Hi + \langle f_i,f_i\rangle_\Hi + i\left(\langle f_i,f_r\rangle_\Hi - \langle f_r , f_i \rangle_\Hi\right)\\&= \norm{f_r}^2_\Hi + \norm{f_i}^2_\Hi + i\left(\langle f_r,f_i\rangle_\Hi - \langle f_r , f_i \rangle_\Hi\right) = \norm{f_r}^2_\Hi + \norm{f_i}^2_\Hi, 
\end{align*}
which is the Pythagorean Theorem.
In cases, where the underlying field of the vector space does not play a role, we denote by $X$ an arbitrary separable real or complex Hilbert space.

\begin{example}
\begin{itemize} \item [(i)]
	Let $\Hi = \R^n$ for $n \in \N$ equipped with the standard Euclidean inner product. Then $\Hi + i\Hi = \R^n + i\R^n = \C^n$ is equipped with the standard sesquilinear inner product.
	\item [(ii)] For this work the Hilbert space is often the space of square-integrable functions. Let $\Hi = \RLtwo$ be the real vector space of all real-valued square-integrable functions. Then $\Hi + i\Hi = \CLtwo$ is the complex function space of all complex-valued square-integrable functions.
	\end{itemize}
\end{example}

We denote by $\mathcal{L}(X)$ the vector space of all bounded linear operators $A: X \to X$, endowed with the operator norm. The spectrum of a linear operator is denoted by $\sigma(A)$, the resolvent set by $\rho(A)$.

We define the Fourier transform of a function $ f \in \CLone$ as \[\F f(\omega) := \int_{\R} f(x) e^{-2\pi i\omega x} \dif x \quad \mbox{for all }\omega\in\R.\] and extend it to $\CLtwo$ in the usual way. Similarly we define the Fourier transform for $L^2(\mathbb{T};\C)$ on the torus $\mathbb{T} := \R / \Z \cong [\, 0,1\,]$  as
\[\hat{\  } : L^2(\mathbb{T};\C) \to \ltwoC,\; \quad \hat{f}(n) := \int_{0}^{1} f(x) e^{-2\pi i n x} \dif x \quad \mbox{for all }n\in\Z,\]
and the inverse Fourier transform as \[\check{\  } : \ltwoC \to L^2(\mathbb{T};\C),\; \quad a = \{a_n :\,n \in \Z\} \mapsto \sum_{n \in \Z} a_n e^{2\pi i n \bullet}.\]

\section{Rebricking Bases}

If $\{f_{n} : n\in\N\}$ is a real-valued basis in $\Hi$, under which conditions on the linear operator $A$ is the rebricked family
$\{f_{n} +i A f_{n} :\, n \in \N\}$ a basis for $\Hi + i \Hi$? In this section we address this question for several types of bases and for infinite-dimensional as well as for finite-dimensional vector spaces. For the first case, we give a characterization of appropriate rebricking operators $A$. For the latter case, we develop new results on the permutation of bases. 

\subsection{Riesz Bases}

As first step we consider the class of Riesz bases and analyze the rebricking.

\begin{defn}[Riesz basis]
	A sequence $\{f_n:\ n \in \N\}$ in a  (real or complex) Hilbert space $X$ is called a Riesz basis  of $X$, if there exists an orthonormal basis $\{e_n:\ n\in \N\}$ and an automorphism $U \in \mathcal{L}(X)$ such that $f_n = Ue_n$ for all $n \in \N$. 
	\end{defn}
	
	As $U$ is a continuous automorphism there exist constants $0 < c < C < \infty$ such that for all $f \in X$ 
	\[c\norm{f}^2 \leq \sum_{n \in \N}\abs{\langle f,f_n\rangle}^2 \leq C \norm{f}^2.\]
The largest possible value for $c$ is $\frac{1}{\norm{U^{-1}}^2}$. The smallest for $C$ is $\norm{U}^2$. 
It is well-known that every Riesz basis is a frame. 
The extremal values of $c$ and $C$ are therefore called \textit{the, strict or sharp} {frame} bounds. 

If two Riesz bases $\{f_n:\ n \in \N\}, \{g_n:\ n \in \N\} \subset X$ allow for a representation \[f = \sum_{n \in \N} \langle f, f_n\rangle g_n\] for all $f \in X$, then they are called a \emph{dual pair} of Riesz bases. For further reading on Riesz bases we refer to \cite[Chapter 3.6]{Christensen2003}.

 As a first observation we show that every Riesz basis for $\Hi$ is also a Riesz basis of the canonically complexified Hilbert space $\Hi + i \Hi$.
\begin{lemma}
	$\{f_n:\ n \in \N\} \subset \Hi$ is a Riesz basis for $\Hi$ if and only $\{f_n:\ n \in \N\} \subset \Hi$ is a Riesz basis for $\Hi + i \Hi$.
\end{lemma}

\begin{proof}
	 According to \cite[Theorem 3.6.6]{Christensen2003}, $\{f_n:\ n \in \N\} \subset \Hi$ is a Riesz basis for $\Hi$ if and only if there exist constants $d,D > 0$ such that for every finite scalar sequence $\{c_n : n \in \N\} \subset \R$  
	 \[d \sum_{n=1}^\infty \abs{c_n}^2 \leq \norm{\sum_{n=1}^\infty c_nf_n}^2 \leq D \sum_{n=1}^\infty \abs{c_n}^2.\]
	 To show that $\{f_n:\ n \in \N\} \in \Hi$ is a Riesz basis for $\Hi + i \Hi$ we suppose $c_n = c^{(r)}_n + i c^{(i)}_n$ is a finite complex sequence decomposed into its real and imaginary part.
	 Then we have 
	 \begin{eqnarray*}
	 d \sum_{n=1}^\infty \abs{c_n}^2 & = & d \sum_{n=1}^\infty \abs{c^{(r)}_n}^2 + d \sum_{n=1}^\infty \abs{c^{(i)}_n}^2 \quad\leq \quad \norm{\sum_{n=1}^\infty c^{(r)}_nf_n}^2 + \norm{\sum_{n=1}^\infty c^{(i)}_nf_n}^2 \\
	 &=& \norm{\sum_{n=1}^\infty \left(c^{(r)}_n + ic^{(i)}_n\right)f_n}^2 = \norm{\sum_{n=1}^\infty c_nf_n}^2.
	 \end{eqnarray*}
	 The upper bound can be proved the same way. The converse is trivial as every real sequence is a complex one.
\end{proof}

Now we wish to answer the question whether the rebricking of two Riesz bases is a Riesz basis again.
I.e. let $\{f_n:\ n \in \N\}, \{g_n:\ n \in \N\} \subset \Hi$ be two Riesz bases; under which conditions is \[\{f_n + ig_n : n \in \N \} \subset \Hi + i\Hi\] a complex Riesz basis? To this end, let $\{e_n:\ n \in \N\} \subset \Hi$ be an orthonormal basis. Let $U_f,U_g \in \mathcal{L}(\Hi)$ be continuously invertible such that $f_n = U_f e_n$ and $g_n = U_g e_n$ for all $n \in \N$. Then the linear operator $A := U_gU_f^{-1}$ maps $g_n = A f_n$ for all $n \in \N$. In particular, $A$ does not depend on the choice of the orthonormal basis $\{e_n : n\in\N\}$, and therefore is also independent of $U_f, U_g$ as $g_n = A f_n$ for all $n \in \N$ determines the linear operator completely, because $\mbox{span} \{f_n : n \in \N \}$ is dense in $\Hi$. Therefore any two Riesz bases are equivalent by virtue of a topological isomorphism $A$.  Equivalently we  can approach the above-raised question by considering \[\{\tilde{f}_n := Bf_n := f_n + i Af_n : n \in \N \},\] and asking under which conditions this system is a Riesz basis of $\Hi + i \Hi$. We do this by investigating the required properties of $A$ resp. $B$ being a topological isomorphism  (and hence $\tilde{f}_n$ being a Riesz basis). It will turn out that this change of view is of great benefit for two reasons:
\begin{enumerate}[label=(\roman*)]
	\item There is a much richer theory for linear operators than for Riesz bases.
	\item Most notably: The conditions on $A$ will be independent of the choice of a specific Riesz basis. Thus a suitable $A$ can be used to rebrick any Riesz basis.
\end{enumerate}

As a first step into the theory we ask under which conditions a bounded linear operator $A$ can be used to build complex Riesz bases from real ones.

\begin{theorem}
	\label{thm_lift_riesz_bases}
	Let $A \in \mathcal{L}(\Hi)$ be an isomorphism. Then, the following three statements are equivalent:
	\begin{enumerate}[label=(\roman*)]
		\item For a real Riesz basis $\{f_n:\ n \in \N\}$ of $\Hi$, the family $\cb{Bf_n := f_n+iAf_n: n\in\N}$ is a complex Riesz basis in $\Hi+i\Hi$.
		\item For every real Riesz basis $\{f_n:\ n \in \N\}$ of $\Hi$, the family $\cb{Bf_n := f_n+iAf_n: n\in\N}$ is a complex Riesz basis in $\Hi+i\Hi$.
		\item $i\in\rho(A)$, i.e., $i$ is in the resolvent set $\rho(A)$ of $A$, hence not an element of the spectrum of $A$.
		\item $-i \in \rho(A^*)$.
	\end{enumerate}
\end{theorem}

\begin{proof}
	$\cb{Bf_n := f_n+iAf_n: n\in\N}$ is a Riesz basis in $\Hi+i\Hi$ if and only if $B = i (A - i\,\id)$ is bounded and invertible on $\Hi + i \Hi$. This is the equivalence of (i), (ii), and (iii).  
	The last two items are equivalent because spectra of two adjoint operators are complex conjugates of each other.
\end{proof}

Notice that an operator $A \in \mathcal{L}(\Hi)$ fulfilling one of the equivalent conditions in \ref{thm_lift_riesz_bases} is independent of the Riesz basis it rebricks, i.e. is able to complexify every Riesz basis via $B = \id +iA$. This observation gives rise to the following definition:

\begin{defn}[Rebrickability]
	\label{D rebrickability Riesz bases}
	\begin{enumerate}[label=(\roman*)]
		\item Let $\{f_n:\ n \in \N\}$, $\{g_n:\ n \in \N\}$ be Riesz bases of $\Hi$. They are called \emph{rebrickable} if $\{\tilde{f}_n := f_n + i g_n : n \in \N \}$ is a Riesz basis of $\Hi + i\Hi$. In this case $\{\tilde{f}_n : n \in \N \}$ is called a complexified (Riesz) basis.
		\item Suppose $A \in \mathcal{L}(\Hi)$. The operator $A$ is called \emph{rebricking operator (for Riesz bases)} if $\{\tilde{f}_n := f_n + i Af_n : n \in \N \}$ is a complex Riesz basis for $\Hi + i \Hi$ for every choice of a Riesz basis $\{f_n \,:\, n\in\N\}$.
	\end{enumerate}	
\end{defn}

Now we explore properties of the rebrickability relation, i.e., the set of pairs of Riesz bases that together can form complexified Riesz bases.  
\begin{lemma}[Properties of the rebrickability relation]
\label{L Riesz case Properties of the rebrickability relation}
	Let $\{f_n:\ n \in \N\}$, $\{g_n:\ n \in \N\}$ and $\{h_n:\ n \in \N\}$ be Riesz bases of $\Hi$. Then the following properties hold true:
	
	\begin{enumerate}[label=(\roman*)]
		\item Reflexivity: $\{f_n + if_n:\ n \in \N\}$ is a complex Riesz basis.
		\item Symmetry:  $\{f_n + ig_n:\ n \in \N\}$ is a complex Riesz basis if and only if $\{g_n + if_n:\ n \in \N\}$ is a complex Riesz basis.
		\item Non-Transitivity: Let $\{f_n + ig_n:\ n \in \N\}$ and $\{g_n + ih_n:\ n \in \N\}$ be complex Riesz bases then $\{f_n + ih_n:\ n \in \N\}$ is \textit{not necessarily} a Riesz basis.
	\end{enumerate}
\end{lemma}	
\begin{proof}
(i) Reflexivity: $\{f_n + if_n:\ n \in \N\} = \{(1+i)f_n :\ n \in \N\}$ is obviously a complex Riesz basis.

(ii) Symmetry: Suppose $\{f_n + ig_n:\ n \in \N\}$ is a complex Riesz basis. Then there is an operator $A \in \mathcal{L}(\Hi)$ such that $g_n = A f_n$ for all $n \in \N$. Thus $f_n = A^{-1}g_n$ for all $n \in \N$ and therefore $\{g_n + if_n:\ n \in \N\}$ is a complex Riesz basis.

(iii) Non-Transitivity: We give a counter example. Let $A_1, A_2 \in \mathcal{L}(\Hi)$ be invertible having the eigenvalue  $\frac{1+i}{\sqrt{2}}$ and the same corresponding eigenspace, but not $i$ as a spectral value. Hence $A_{1}$ and $A_{2}$ are rebricking operators. Now let $g_n := A_1 f_n$ and $h_n := A_2 g_n$ for all $n \in \N$. Clearly $\{f_n + ig_n:\ n \in \N\}$ and $\{g_n + ih_n:\ n \in \N\}$ are complex Riesz bases, but $\{f_n + ih_n:\ n \in \N\}$ is not because $h_n = A_2A_1f_n$ for all $n \in \N$, however, $A_2A_1$ having an eigenvalue $\frac{1+i}{\sqrt{2}}\cdot \frac{1+i}{\sqrt{2}} =i$ is not a rebricking operator.
\end{proof}

In general it is a non-trivial task to compute the spectrum of an invertible and bounded linear operator $A$, or to only check whether $i \in \sigma(A)$, i.e., whether $i$ is in the spectrum of $A$, or not. But there are some special cases which simplify the situation a lot.
\begin{lemma}
	Let $A \in \mathcal{L}(\Hi)$ be invertible. Then $A$ is a rebricking operator for Riesz bases if one of the following sufficient conditions holds: 
\begin{enumerate}[label=(\roman*)]
	\item $A$ is self-adjoint. 
	\item The operator norm of $A$ is bounded by 1: $\norm{A} < 1$.
\end{enumerate}  	
\end{lemma}

\begin{proof}
	\begin{enumerate}[label=(\roman*)]
		\item $A$ is self-adjoint, hence the spectrum is real.
		\item If $\norm{A} < 1$ then the spectral values of $A$ have a modulus smaller than one.
	\end{enumerate} 
\end{proof}

It is also possible to transform the complex problem into a real one.
\begin{lemma}
	\label{lem:complex-bijective-to-real-bojective}
	Let $A \in \mathcal{L}(\Hi)$ be an isomorphism. Then the following are equivalent: 
	\begin{enumerate}[label=(\roman*)]
		\item $\id + iA$ is bijective in $\Hi + i\Hi$.
		\item $\id - iA$ is bijective in $\Hi + i\Hi$.
		\item $\id + A^2$ is bijective in $\Hi$.
	\end{enumerate}
\end{lemma}

\begin{proof}
	(i) $\Leftrightarrow$ (ii):	
	$A$ is a real linear operator, and the spectrum of real operators is invariant under complex conjugation.
	
	(i) $\Rightarrow$ (iii): 
	$\id + A^2 = (\id +iA)(\id -iA)$. Thus $\id + A^2$ is bijective because it is a composition of two bijective operators.
	
	(i) $\Leftarrow$ (iii):
	We assume that $\id + iA$ is not bijective in $\Hi + i\Hi$. Then $\id + A^2 = (\id +iA)(\id -iA)$ is not bijective in $\Hi + i\Hi$. $\id + A^2$ is a real operator, therefore the real and imaginary part in $\Hi + i\Hi$ can be mapped independently. Thus $\id + A^2$ is not bijective in $\Hi$. This is a contradiction.
\end{proof}

\begin{example}[The Hilbert transform]
	The Hilbert transform $H$ on the real line is defined as Cauchy-Principle Value of the integral over a singular function
	$$
	H: \RLtwo \to \RLtwo,\; Hf(x) = \frac{1}{\pi} P.V. \int_{\R} \frac{f(y)}{x-y}\, dy \quad \mbox{for almost all } x\in\R.
	$$
	A pair of functions $f,g \in \RLtwo$ satisfying $g = Hf$ constitutes a so-called \emph{Hilbert transform pair} 
	\cite[Sec. 1.3]{kingvol1}.
	For $f\in \CLtwo$, the Hilbert transform is an isometric isomorphism: 
	$$\norm{f}_{2}  = \norm{Hf}_{2}, $$
	satisfies the inversion property
	$$
	H^{2}f = -f
	$$
	By \ref{lem:complex-bijective-to-real-bojective} we can see that $H$ has $\pm i$ as spectral value and therefore it is not suitable as a rebricking operator.
	
	Let $\{f_{n} : n\in\N\} \subset \RLtwo$ be a Riesz basis.
	As $H$ is an isometric isomorphism, $\{ H f_{n} : n\in\N\}$ is also a Riesz basis of $\RLtwo$, and $f_{n}$ and $H f_{n}$ form Hilbert pairs for every $n$. 
	However $\{f_{n} + i H f_{n} : n\in\N\}$ is not a Riesz basis for $\CLtwo$. Its span is not even dense in $\CLtwo$, because the operator $\id +i H$ possesses a non-trivial kernel:
	$$
	(\id + iH)  (f_{n} - i Hf_{n}) = f_{n} + i H f_{n} -i Hf_{n} - f_{n} = 0 
	$$
	for every real basis element $f_{n}$.
	
	Consequently, considering the analytic signal $(\id + i H) f$  for all real-valued function $\RLtwo$ does not generate the whole complex space $\CLtwo$.
	\end{example}

For a Riesz basis $\{ f_n : n \in \N \} \subset \Hi$ (analysis basis) the dual Riesz basis $\{ g_n : n \in \N \} \subset \Hi$ (synthesis basis) is needed to represent $f\in \Hi$: 
$$f = \sum_{n \in \N} \langle f, f_n \rangle g_n.$$
In general, the analysis Riesz basis is different from the synthesis Riesz basis. Therefore, it is desirable to rebrick the basis and its dual basis simultaneously, and in the same way. Special case: Riesz bases whose analysis basis and synthesis basis coincide are indeed orthonormal bases.

\begin{lemma}\label{L rebricked dual Riesz basis}
	Let $\{f_n:\ n \in \N\} \subset \Hi$ and $\{g_n:\ n \in \N\} \subset \Hi$ be a pair of dual Riesz bases. Suppose $A\in \mathcal{L}(\Hi)$ is an isomorphism and a rebricking operator. Let $B= \id + iA$. Then the rebricked Riesz basis $\{Bf_n : n \in \N \}  \subset \Hi + i\Hi $ has a rebricked dual Riesz basis $\{\left(B^*\right)^{-1}g_n : n \in \N \} \subset \Hi + i\Hi$.
\end{lemma}

\begin{proof}
	Let $\{f_n= U e_n:\ n \in \N\} \subset \Hi$ be an arbitrary Riesz basis generated via some invertible $U\in \mathcal{L}(\Hi)$ from an orthonormal basis $\{e_n:\ n \in \N\} \subset \Hi$. It is well-known that its unique dual Riesz basis is given by $g_n := \left(U^*\right)^{-1}e_n$, see \cite[Sec. 3.6]{Christensen2003}.
Thus the rebricked dual basis can be computed by
\[\left(\left(BU\right)^*\right)^{-1}e_n = \left(U^*B^*\right)^{-1}e_n = \left(B^*\right)^{-1}\left(U^*\right)^{-1}e_n = \left(B^*\right)^{-1}g_n.\]
\end{proof}

We see from this Lemma that for rebricking a given pair of basis and dual-basis the operator $U$ and the associated orthonormal basis are not required. It suffices to compute $\left(B^*\right)^{-1}$ which can be hard enough in some cases. In general, we cannot expect $\left(B^*\right)^{-1}$ to have a similar structure as $B = \id + iA$. However, in this sense, the linear operator $B$ is enough to rebrick the pair. 

Let $\{f_n:\ n \in \N\} \subset \Hi$ and $\{g_n:\ n \in \N\} \subset \Hi$ be a pair of dual Riesz bases and  $\{\tilde{f}_n:\ n \in \N\} \subset \Hi +i\Hi$ and $\{\tilde{g}_n:\ n \in \N\} \subset \Hi + i\Hi$ be the rebricked pair of dual Riesz bases. Then the real part $\Re \tilde{f}_n = f_n$. However, in general it is not true that the rebricked dual Riesz basis fulfills $\Re \tilde{g}_n = g_n$. The reason is that $\left(B^*\right)^{-1}$ in general does  not preserve the real part. In fact, the real part is never preserved exactly, but only up to a constant, as the following Lemma shows.

\begin{lemma}
	\label{lemma:dual_frame_preserve_real_part}
	Let $\lambda \in  \R \setminus \{0,1\}$. Under the same conditions as in Lemma \ref{L rebricked dual Riesz basis}  the following are equivalent:
	\begin{enumerate}[label=(\roman*)]
		\item $\Re\left( (B^*)^{-1} g_n\right) = \lambda g_n$ for all $n \in \N$.
		\item $A^{-1} = \frac{\lambda}{1-\lambda}A$.
	\end{enumerate}
\end{lemma} 

\begin{proof}
Denote $\tilde{g}_n = (B^*)^{-1} g_n$. We suppose (i). Then  $$ (B^*)^{-1}g_n = \Re\tilde{g}_n + i\Im\tilde{g}_n = \lambda g_n +i\underbrace{\Im\tilde{g}_n}_{\displaystyle :=u_n}.$$ Moreover,  \begin{align*}
	&g_n = B^* \tilde{g}_n =  B^*(\lambda g_n + iu_n) = (\id-iA^*)(\lambda g_n + iu_n) = \lambda g_n + iu_n - i\lambda A^*g_n + A^*u_n.\end{align*} This implies
	\begin{align*}(1-\lambda) g_n + i\lambda A^*g_n = A^*u_n + iu_n\end{align*} which is the same as
	\begin{align}(1-\lambda)\left(A^*\right)^{-1}g_n = u_n \quad \mbox{and} \quad \lambda A^*g_n = u_n.
	\label{eq u n im Riesz Basen Beweis}
	\end{align}
	Thus \[ (1-\lambda)\left(A^*\right)^{-1}g_n = \lambda A^*g_n.\]
	As $\Span \{g_n : n \in \N\}$ is dense in $\Hi$ this is equivalent to $ \frac{\lambda}{1-\lambda}A^* = \left(A^*\right)^{-1}$. Hence $\frac{\lambda}{1-\lambda}A = A^{-1}$, which is (ii).
	
	For the converse, apply the arguments in the reverse order. We define $u_{n}$ as in \eqref{eq u n im Riesz Basen Beweis}. Hereby note that $u_n = A^*g_n \in \Hi$.
\end{proof}
 
In addition to the dual Riesz basis, which is an exact frame, the frame bounds of the rebricked Riesz basis are of particular interest. To this end let $\{f_n := Ue_n : n \in \N\} \subset \Hi$ be a Riesz basis for some orthonormal basis $\{e_n : n \in \N\} \subset \Hi$ and an isomorphism $U \in \mathcal{L}(\Hi)$. Moreover let 
$$c:=\frac{1}{\norm{U^{-1}}^2}, \quad C := \norm{U}^2$$ be the lower and upper frame bound respectively. Let $\{\tilde f_n := Bf_n : n \in \N\} {\subset \Hi + i\Hi}$ be the rebricked Riesz basis for a linear operator $B$ as described in Theorem \ref{thm_lift_riesz_bases}.

By definition the rebricked Riesz basis' lower frame bound is $\displaystyle\tilde{c} = \frac{1}{\norm{(BU)^{-1}}^2}$ and the upper frame bound is $\tilde{C} = \norm{BU}^2$. Unfortunately $U$ is not necessarily known for a Riesz basis, however the exact frame bounds or at least estimates are typically known. Given the bounds we can estimate the rebricked Riesz basis bounds by two simple calculations:
\[\tilde{c} = \frac{1}{\norm{(BU)^{-1}}^2} = \frac{1}{\norm{U^{-1}B^{-1}}^2} \geq \frac{1}{\norm{U^{-1}}^2\norm{B^{-1}}^2} = c\,\norm{B^{-1}}^{-2} \] for the lower bound and \[\tilde{C} = \norm{BU}^2 \leq \norm{B}^2 \norm{U}^2 = \norm{B}^2C.\] for the upper bound.
The operator norm of $B$ can be directly computed from the operator norm of $A$ by \[\norm{B} = \sup_{\norm{f} = 1} \sqrt{\norm{f}^2 + \norm{Af}^2} = \sqrt{1+\norm{A}^2}.\]

\subsection{Orthonormal Bases}
\label{ssec Orthonormal Bases}

Let $\Hi$ be a separable real Hilbert space and let $\{e_n:\ n\in\N\}$ be an orthonormal basis for $\Hi$. Our goal is again to answer the question whether two orthonormal bases can be used to rebrick a complex orthonormal basis of the related complex Hilbert space $\Hi+i \Hi$. To this end, we observe that two orthonormal bases are equivalent up to some \emph{unitary} operator $A:\Hi\to\Hi$. Therefore we choose the analogous ansatz
\[\cb{\tilde{e}_n:= Be_n := \frac 1{\sqrt{2}} (e_n+i Ae_n) : n\in\N}\] and ask for conditions on $B$.

\begin{lemma}\label{L e n ortho basis in H und H iH}
	$\{e_n:\ n\in\N\} \subset \Hi$ is an orthonormal basis for $\Hi$ if and only if $\{e_n:\ n\in\N\} \subset \Hi$ is an orthonormal basis for $\Hi + i\Hi$.
\end{lemma}

\begin{proof}
	Orthonormality and completeness of the real basis are maintained both for the complex as well as for the restriction on the real part.
\end{proof}

\begin{theorem}
\label{S A unitaer und selbstadjungiert fuer ONB}
Suppose $\{e_n:\ n\in\N\} \subset \Hi$ is an orthonormal basis for $\Hi$.
For $A \in \mathcal{L}(\Hi)$
	 the following statements are equivalent:
	\begin{enumerate}[label=(\roman*)]
		\item The system \[\cb{\tilde e_n:= Be_n := \frac 1{\sqrt{2}} ( e_n+iA e_n):\ n\in\N}\] is an orthonormal basis in $\Hi+i \Hi$.
		\item $A:\Hi\to\Hi$ is self-adjoint and unitary. 
	\end{enumerate}
	In this case it is $A = A^* = A^{-1}$.
\end{theorem}

\begin{proof}
	$\cb{\tilde e_n :\ n\in\N}$ is an orthonormal basis if and only if $B$ is unitary on $\Hi+i\Hi$. To this end consider the equality 
	\begin{align*}BB^* &= \frac{1}{2} (\id+i A) (\id-i A^*)= \frac{1}{2} (\id + AA^*) +\frac{i}{2} (A-A^*) \overset{!}{=} \id\end{align*} and similarly \[B^*B = \frac{1}{2}\left(\id +A^*A\right)+\frac{i}{2}\left(A^*-A\right) \overset{!}{=}\id.\]

	(i) $\Rightarrow$ (ii): From the two equations above we see that then $A=A^{\ast}$ and $AA^{\ast} = \id = A^{\ast}A$. Hence $A$ is self-adjoint and unitary.

	(ii) $\Rightarrow$ (i): Conversely if $A$ is unitary and self-adjoint, then from the same two equations above we see that the operator $B$ is unitary. 

	If $A$ is self-adjoint and unitary $AA^{-1} = \id = A^*A = A^2$, i.e. $A = A^{-1}$.
\end{proof}

Finally we return to the initial question whether two real orthonormal bases can be used for rebricking a complex orthonormal basis.
\begin{proposition}
    Suppose $\{e_n:\ n\in\N\}, \{{d}_n:\ n\in\N\} \subset \Hi$ are orthonormal bases. Then there are equivalent:
    \begin{enumerate}[label=(\roman*)]
        \item The system \[\cb{\tilde{e}_n := \frac 1{\sqrt{2}} ( e_n+i {d}_n):\ n\in\N}\] is an orthonormal basis in $\Hi+i \Hi$.
        \item For all $n,k \in \N$ holds the symmetry condition \[\langle {d}_n, e_k\rangle = \langle {d}_k, e_n \rangle.\]
    \end{enumerate}
\end{proposition}

\begin{proof} We start with some preliminary thoughts:
    \[\cb{\tilde{e}_n := \frac 1{\sqrt{2}} ( e_n+i {d}_n):\ n\in\N}\] is an orthonormal basis in $\Hi+i \Hi$ if and only if the operator $A: \Hi \to \Hi$,   $Ae_n := {d}_n$ is a unitary, self-adjoint operator. In our case, $A$ is unitary because it maps an orthonormal basis onto an orthonormal basis. The unitary operator $A$ is self-adjoint if and only if $A^2 = \id$, which is in turn equivalent to $A^2 e_n = e_n$ for all $n\in\N$. By inserting the definition of $A$ we get
    \[A^2e_n = A\,{d}_n = A \sum_{k \in \N} \langle {d}_n, e_k\rangle e_k=  \sum_{k \in \N}\langle {d}_n, e_k\rangle Ae_k = \sum_{k \in \N}\langle {d}_n, e_k\rangle {d}_k\] for all $n \in \N$.
Because $\{{d}_n:\ n\in\N\}$ is a real orthonormal basis it automatically holds
    \[e_n = \sum_{k \in \N} \langle e_n, {d}_k\rangle {d}_k = \sum_{k \in \N} \langle {d}_k, e_n \rangle {d}_k\] for all $n \in \N$.

(i) $\Rightarrow$ (ii): Suppose (i). Then it is \[\sum_{k \in \N}\langle {d}_n, e_k\rangle {d}_k = A {d}_n = A^2 e_n = e_n =  \sum_{k \in \N} \langle {d}_k, e_n \rangle {d}_k\] for all $n \in \N$. I.e. for all $n,k \in \N$ we get $\langle {d}_n, e_k\rangle = \langle {d}_k, e_n \rangle$.

    (ii) $\Rightarrow$ (i):
    Suppose (ii), $\langle {d}_n, e_k\rangle = \langle {d}_k, e_n \rangle$ for all $n,k \in \N$.
    This yields \[A^2e_n = \sum_{k \in \N}\langle {d}_n, e_k\rangle {d}_k = \sum_{k \in \N}\langle {d}_k, e_n \rangle {d}_k= \sum_{k \in \N}\langle e_n, {d}_k \rangle {d}_k = e_n.\] 
    Thus $A^2=\id$. According to the preliminary thoughts this is equivalent to (i).
\end{proof}

\begin{remark}[Comparison to the Riesz basis case]
	Notice that---in contrast to the Riesz basis case---here $A$ being self-adjoint is not sufficient but only necessary.
	For self-adjoint $A$, the spectrum is always real, hence $i$ is never an element.
	
	Orthonormal bases have the useful property that the dual Riesz basis is the orthonormal basis itself. Thus the real part of the dual orthonormal basis equals the initial real part up to a factor of $\frac{1}{\sqrt{2}}$. 
	
	Here are two examples that illustrate that the conditions self-adjoint and unitary can not be relaxed without losing the orthonormal basis property:
	\begin{itemize}
		\item[a)] Suppose $A$ is self-adjoint and invertible but not unitary. Then $\{Ae_n : n \in \N\}$ is a Riesz basis and $\{\frac{1}{\sqrt{2}}(e_n +iAe_n) : n \in \N\}$ is a complex Riesz basis whose real part is an orthogonal basis. 
		\item[b)] Suppose $A$ is unitary without spectral value $i$ but not self-adjoint. Then $\{Ae_n : n \in \N\}$ is an orthonormal basis and $\{\frac{1}{\sqrt{2}}(e_n +iAe_n) : n \in \N\}$ is a complex Riesz basis whose real and imaginary parts is are orthogonal bases.
	\end{itemize}
\end{remark}

%
%

\begin{example}[Rebricked orthonormal bases of translates in $\ltwoC$]
\label{ex Rebricked orthonormal bases of translates in ltwo}
As first example for the rebricking of orthonormal bases we now examine bases of translates in $\ltwoR$:
\[\{T^nx :\ n \in \Z\},\] where $x \in \ltwoR$, $ n \in \Z$, and 
\[T : \ltwoR \to \ltwoR,\quad (x_{m})_{m\in\Z} \mapsto (x_{m-1})_{m\in\Z}\]
is the right shift operator.

(Note that  for the space $\CLtwo{}$ no such basis exists \cite[Theorem 9.2.1]{Christensen2003}. This is why enhanced approaches of complete systems generated by translates are needed in this case, e.g. wavelets or Gabor frames.)

To rebrick the orthonormal basis $\{T^nx:\ n \in \Z\}$ of $\ltwoR$ we require in accordance with Theorem \ref{S A unitaer und selbstadjungiert fuer ONB} that $A: \ltwoR \to \ltwoR$ is unitary and self-adjoint. Furthermore we require $AT = TA$ as an additional property. This ensures that $\{AT^n x = T^nAx:\ n \in \Z\}$ is also an orthonormal basis of shifts and so is the rebricked basis $$\Big\{\frac{1}{\sqrt{2}}\left(\id + iA\right)T^n x = T^n\frac{1}{\sqrt{2}}\left(Id + iA\right)x:\ n \in \Z\Big\}.$$

It is well-known that bounded linear shift-invariant operators have a convolution structure. Denote $\{e_{n} : n\in\Z\}$ the canonical basis of $\ltwoR$ (and of $\ltwoC$),  and  $a = Ae_{0} \in \ltwoR$. Then for all $x \in \ltwoR$ when applying $TA = AT$,
\begin{eqnarray*}
Ax & = & A(x \ast e_{0})  \quad = \quad A\left(\sum_{n=-\infty}^{\infty} x_{n} e_{0}(\bullet -n)\right)
\quad = \quad A\left(\sum_{n=-\infty}^{\infty} x_{n} T^{n}e_{0}\right) \\
& = & \sum_{n=-\infty}^{\infty} x_{n} AT^{n}e_{0} \quad = \quad  \sum_{n=-\infty}^{\infty} x_{n} T^{n}Ae_{0} \quad = \quad x \ast a.
\end{eqnarray*}
We can rewrite this with the Fourier transform and the convolution theorem as

\[Ax = \left( m \cdot \check{x}\right)\hat{\ }\]
where $m = \check{a} \in L^2(\mathbb{T};\C)$ is the corresponding multiplier. 
\begin{enumerate}[label=(\roman*)]
\item As $A$ is self-adjoint, $m$ is real-valued, because for all real sequences $x,y$,
$$
 \int_{\mathbb{T}}m(t) \,\check{x}(t) \,\overline{\check{y}(t)}\, dt =\langle Ax, y\rangle =  \langle x, Ay\rangle 
 =\int_{\mathbb{T}}\check{x}(t)\, \overline{m(t)}\, \overline{\check{y}(t)}\, dt 
$$
as the Fourier transform is an isometrical isomorphism. 
\item As $A$ is unitary, $m$ has to preserve the norm:  
\begin{eqnarray*}
\norm{\check{x}}_{L^{2}}^{2}= \norm{x}_{\ell^2}^{2} = \langle A^{\ast}Ax,x\rangle = \langle Ax, Ax\rangle = \langle m \,\check{x}, m \,\check{x}\rangle_{L^{2}} = \norm{m\,\check{x}}_{L^{2}}
\end{eqnarray*}
for all $x \in \ltwoR$. Hence $\abs{m(t)}=1 $ for almost all $t \in \mathbb{T}$.
\item  The operator $A$ maps real valued sequences to real valued  sequences. Hence for real-valued $x$ we get from the properties of the Fourier transform
$\check{x}(-t) = (\overline{x})\check{\ }(t) = \overline{\check{x}(t)}$ and similarly $\check{Ax}(-t) =  \overline{\check{Ax}(t)}$. Hence for all $x\in \ltwoR$,
$$
m(-t)\,\overline{\check{x}(t)} = m(-t) \,\check{x}(-t) = (Ax)\check{\ } (-t) = \overline{(Ax)\check{\ } (t)} = \overline{m(t) \,\check{x}(t)} = m(t) \,\overline{\check{x}(t)},
$$
because by (i) $m$ is real-valued. Consequently, $m$ is symmetric: $m(t) = m(-t)$ for almost all $t \in\mathbb{T}$.
\end{enumerate}

We conclude that $m$ is even, real-valued, and $m(t) \in \{ -1,1\}$ for almost all $t\in\mathbb{T}$.
In summary the operator $B$ is of the form \[Bx = \frac{\id+iA}{\sqrt{2}}x = \left( \frac{1+i\,m}{\sqrt{2}} \check{x}\right)\hat{\ }.\] In this sense $B$ is diagonalized by the Fourier transform; see also for  
Example \ref{ex:finit_dim_geom_interpretation} to compare with the diagonalization in finite dimensional spaces.
\end{example}

\begin{example}[Rebricked trigonometric basis] 
Consider the space $L^2(\mathbb{T};\R)$ endowed with the inner product
$\langle f,g \rangle : = \int_{0}^{1} f(x) \, g(x) \, dx$ and  the classical real-valued trigonometric orthonormal basis 
$$B = \{1, \sqrt{2}\cos(2 \pi x),  \sqrt{2}\sin(2 \pi x), \sqrt{2}\cos(4 \pi x), \sqrt{2}\sin(4 \pi x), \ldots\}$$
consisting in the ordering of the constant function $1$ and then alternating $\cos$ and $\sin$-terms of the form $\sqrt{2} \cos (2 \pi k x)$, 
$\sqrt{2} \sin (2 \pi k x)$, $k\in\N$.
For 
$$
C = \{1, \sqrt{2}\sin(2 \pi x),  \sqrt{2}\cos(2 \pi x), \sqrt{2}\sin(4 \pi x), \sqrt{2}\cos(4 \pi x), \ldots\}
$$
we chose the same orthonormal basis, but with a different ordering:
We start with $1$ and then alternating $\sin$ and $\cos$-terms of the form $\sqrt{2} \sin (2 \pi k x)$, 
$\sqrt{2} \cos (2 \pi k x)$, $k\in\N$.
With those we perform the orthonormal rebricking approach $(b_{j} + i\, c_{j})/\sqrt{2}$, $j\in \N_{0}$, with $b_{j} \in B$ and $c_{j}\in C$ being the $j$th element in the basis, respectively. 

For $j=0$, $(b_{0} + i\, c_{0})/\sqrt{2} = \frac{1+i}{\sqrt{2}}$.

For $j\in \N$, $j$ odd, we represent $j= 2 k- 1$, $k \in \N$. Then $(b_{j} + i\, c_{j})/\sqrt{2} =  \cos (2 \pi k x) + i\sin (2 \pi k x) = \exp(2\pi i k x)$.

For $j\in \N$, $j$ even, we represent $j= 2 k$, $k \in \N$. Then $(b_{j} + i\, c_{j})/\sqrt{2} =  \sin (2 \pi k x) + i\cos (2 \pi k x) = i \exp(- 2\pi i k x)$. Hence with the rebricking approach, we retrieve the complex exponential basis of $L^2(\mathbb{T};\C)$, up to multiplicative phase factors.
We refer to \cite{bachman} for an introduction to Fourier series for real-valued functions and for shedding light on the historical developments.
\end{example}

\subsection{Rebricking in finite-dimensional vector spaces}

In the finite-dimensional case Riesz bases are ordinary vector space bases. The linear operator $A$ for rebricking bases can be expressed as a real matrix $A \in \R^{n\times n}$. Furthermore spectrum and eigenvalues coincide and we have that $\sigma(A^T) = \sigma(A)$. In the following, we represent the basis via the columns of an invertible matrix of the form 
\begin{equation}V = (v_{1} | v_{2} | \ldots | v_{n}),
\label{eq Basis in Matrix}
\end{equation}
where the columns $v_{k}$, $k=1,\ldots,n$, are the basis vectors. 

In this setting, we can deduce the following result from Theorem \ref{thm_lift_riesz_bases}:

\begin{corollary}[Finite-dimensional basis rebricking]
\label{cor Finite dimensional basis rebricking}
	Let $V_1,V_2 \in \gl(n,\R)$ whose columns form bases of $\R^n$. Then the following are equivalent:
	\begin{enumerate}[label=(\roman*)]
		\item $V_1 + iV_2 \in GL(n, \C)$, i.e., the columns form a basis for $\C^n$.
		\item $A := V_2V_1^{-1}$ has no eigenvalue $i$. 
		\item $A$ is a rebricking operator and it is $V_1 + iV_2 = V_1 + iAV_2$.
	\end{enumerate}
\end{corollary}

\begin{remark}
	\begin{enumerate}[label=(\roman*)]
		\item In contrast to the infinite-dimensional case it is possible to  verify algorithmically  whether two bases set in the columns of $V_1,V_2 \in \R^{n\times n}$ can be rebricked, e.g. via the Singular Value Decomposition SVD \cite{trefethen97}.
		\item From a probabilistic point of view we observe that given two random real bases in a finite-dimensional vector space, the probability that the rebricking yields a complex basis is $1$ because the probability of having an eigenvalue $i$ is zero.
	\end{enumerate}
\end{remark}

\begin{example}[An illustrating counter example]
	\label{ex:counter_example}
	In the finite-dimensional case we can give a geometric interpretation to the rebricking matrices not having the eigenvalue $i$.
To this end, consider the three regular matrices 
	\[V_1 = \begin{pmatrix}	1 & 0\\0 & 1
	\end{pmatrix}, \quad V_2 = \begin{pmatrix}
	\cos(\frac{\pi}{4}) & -\sin(\frac{\pi}{4})\\[0.8ex]
	\sin(\frac{\pi}{4}) & \cos(\frac{\pi}{4})
	\end{pmatrix} = \begin{pmatrix}\frac{1}{\sqrt{2}} & -\frac{1}{\sqrt{2}}\\[0.8ex] 
	\frac{1}{\sqrt{2}} & \frac{1}{\sqrt{2}}\end{pmatrix}, \quad V_3 = \begin{pmatrix}
	0 & -1\\1 & 0
	\end{pmatrix}.\]
	
	in $\gl(2,\R)$. Then 
	\begin{align*}A_{12} = V_2V_1^{-1} = \begin{pmatrix}
	\cos(\frac{\pi}{4}) & -\sin(\frac{\pi}{4})\\\sin(\frac{\pi}{4}) & \cos(\frac{\pi}{4})
	\end{pmatrix} = \begin{pmatrix}\frac{1}{\sqrt{2}} & -\frac{1}{\sqrt{2}}\\[0.8ex] 
	\frac{1}{\sqrt{2}} & \frac{1}{\sqrt{2}}\end{pmatrix},
	\end{align*}
	\begin{align*}
	A_{23} = V_3V_2^{-1} = \begin{pmatrix}
	\sin(\frac{\pi}{4}) & -\cos(\frac{\pi}{4})\\\cos(\frac{\pi}{4}) & \sin(\frac{\pi}{4})
	\end{pmatrix} = \begin{pmatrix}\frac{1}{\sqrt{2}} & -\frac{1}{\sqrt{2}}\\[0.8ex] 
	\frac{1}{\sqrt{2}} & \frac{1}{\sqrt{2}}\end{pmatrix}, \quad \mbox{and}
	\end{align*}
	\begin{align*}
	A_{13} = V_3V_1^{-1} = \begin{pmatrix}
	0 & -1\\1 & 0
	\end{pmatrix}.
	\end{align*}
	While the first two matrices have the eigenvalues $\frac{1\pm i}{\sqrt 2}$, the last has the eigenvalues $\pm i$, and indeed 
	
	\[V_1 + iV_3 = \begin{pmatrix}
	1 & -i\\
	i & 1
	\end{pmatrix} \] 
	is not regular in $GL(2,\C)$.
	
	The counterexample shows that rotations by $\frac{\pi}{2}$ on subspaces are the kind of geometrical actions that are not allowed for rebricking operators $A$.
\end{example}

Conversely one might wonder whether a complex basis can be split up into two real bases via their real and imaginary parts. Unfortunately complex bases need not even have an imaginary part for all basis vectors and the ones which have imaginary parts need not have linear independent ones. As a counter example  consider the regular matrix
\[\begin{pmatrix}1 + i & i & 0\\ 0 & 1 & 0\\ 0& 0& 1\end{pmatrix}.\] Not all column vectors have an imaginary part and those who have share all the same one. In fact, this example falls in the classes of bases which we can not construct by our rebricking approach, and clearly, most bases are of this non-bricked form.

In the finite-dimensional setting orthonormal bases form orthogonal matrices via the representation \eqref{eq Basis in Matrix}. The finite-dimensional equivalent of real unitary self-adjoint operators are orthogonal symmetric matrices. Therefore we can reformulate the general orthonormal basis rebricking Theorem \ref{S A unitaer und selbstadjungiert fuer ONB} as follows:
\begin{corollary}[Finite-dimensional orthonormal basis rebricking]
	Suppose that two orthonormal bases of $\R^n$ are represented by the columns of the orthogonal matrices $E_1,E_2 \in \R^{n\times n}$. Then the following are equivalent:
	\begin{enumerate}[label=(\roman*)]
		\item The columns of $\frac{1}{\sqrt{2}}(E_1 + iE_2)$ form an orthonormal basis for $\C^n$.
		\item $A := E_2E_1^{-1} = E_2E_1^{T}$ is symmetric. 
	\end{enumerate}
\end{corollary}

\begin{example}[Geometric interpretation]
	\label{ex:finit_dim_geom_interpretation}
	Now we investigate the geometric operation which an orthogonal, self-adjoint matrix $A \in \R^{n \times n}$  performs on orthonormal bases. Due to its symmetry $A$ has only real eigenvalues, and as $A$ is orthogonal they are $\pm 1$. By the spectral theorem of linear algebra there exists an orthogonal matrix  $R \in \R^{n\times n}$ whose columns are a basis of eigenvectors such that 
	\begin{equation}R^TAR = \begin{pmatrix}
	1 & & &&&\\
	& \ddots & &&&\\
	& & 1 &&&\\
	&&& -1 &&\\
	&&&& \ddots &\\
	&&&&& -1\\
	\end{pmatrix} =: D.
	\label{eq matrix D}
	\end{equation}
	
	Clearly the converse is also true: For every $R \in SO(n)$ and every matrix $D$ of the form \eqref{eq matrix D} the product $A = RDR^T$ yields an orthogonal symmetric matrix. Thus given an orthogonal matrix $E \in \R^{n\times n}$ every ``bricked'' orthogonal matrix $\Phi = \frac{ 1}{\sqrt{2}} (\id+i A) E$ is built using a rotation matrix $R$ and a diagonal reflection matrix $D$: \[\Phi = \frac{1}{\sqrt{2}} \left(\id + i RDR^T\right)E = \frac{1}{\sqrt{2}} R \left(\id + i D\right)R^TE.\]
	
	Interpretation: To get a rebricked basis, every column vector of $E$ has to be rotated by $R^T\in SO(n)$; then it is rebricked via a multiplication by $\frac{1 \pm i}{\sqrt{2}}$. This multiplication corresponds to a rotation with $\pm \frac{\pi}{4}$ in a two-dimensional subspace of $\R^{2n}\cong\C^{n}$. Finally the result is rotated back by $R$.
\end{example}

Until now we have always asked the question under which circumstances the rebricked system is a basis again. Now we study the question what can be done if two real bases do not generate a complex basis. We have already found that the set of matrices $A\in\R^{n\times n}$ satisfying that for every basis $\{v_1,\ldots,v_n\}$ of $\R^n$ the set $\{v_1+iAv_1,\ldots,v_n+iAv_n\}$ is a basis of $\C^n$ consists of all $A\in \gl(n,\R)$ with $i$ not being an eigenvalue of $A$, see Corollary \ref{cor Finite dimensional basis rebricking}.

Let us consider again the counter example \ref{ex:counter_example} with $V_1$, $V_3$, and $A = A_{13}$. Yet, by switching the order of the basis vectors in exactly one of the two matrices by a permutation matrix,  $P=\begin{pmatrix}
0 & 1 \\
1 & 0
\end{pmatrix}$, we get when e.g. permuting $V_{3}$: \[V_1 + iV_3P = V_1+iAV_1P = \begin{pmatrix}
1 -i& 0\\
0 & 1+i
\end{pmatrix}. \]
Hence the column vectors indeed form a basis of $\C^2$.

In more general terms: Suppose that $V \in \gl(n,\R)$ contains the basis in its columns and $A \in \gl(n,\R)$ is not a rebricking operator:  $V+iAV\notin \gl(n,\C)$. As $V$ is regular, this is equivalent to $\id+iA\notin \gl(n,\C)$.
As we intend to work on $V$ via permutations, we perform a change of basis with respect to $V$: $\tilde A=V^{-1}AV$.

Now we ask whether there exists a permutation matrix $P$ such that
$$ \id+i\tilde AP\in \gl(n,\C),$$ 
i.e., $\tilde AP$ is a rebricking operator?

Before we approach said permutation problem, we revisit a basic result about characteristic polynomials.

The characteristic polynomial of $\tilde A$
$$
\chi_{\tilde A}(\lambda) = \det(\lambda\id - \tilde A) = \lambda^{n} + c_{n-1}(\tilde A) \lambda^{n-1} + c_{n-2}(\tilde A) \lambda^{n-2} + \ldots + c_{0}(\tilde A)
$$
has the coefficients
\begin{equation}\label{coeff}
	c_k(\tilde A) = (-1)^{n-k} \sum_{\substack{I\subset [n] \\ |I|=k}} \det(\tilde A_I), \quad k\in\{0,\ldots,n-1\},
\end{equation}
where $[n]:=\cb{1,\ldots,n}$ and $\tilde A_I$ is the matrix $\tilde A$ where the rows and columns with indices in $I$ have been deleted. 

As we will permute the columns of $\tilde A$, we need to introduce a more general notation. 

\begin{defn}
	For $I,J\subset[n]$, $|I|=|J|=m< n$,  $\tilde A_{I,J}$ is the matrix that results from deleting  the rows with indices in $I$ and the columns with indices in $J$ of $\tilde A$. In the case $I = J$ we abbreviate $\tilde A_I=\tilde A_{I,I}$. In the case  $I=\{i\}$, $J=\{j\}$, we write for short $\tilde A_{i,j}$.
\end{defn}

\begin{defn}
Let $S_{n}$ denote the group of permutations of the $[n]$ and let $\pi\in S_n$.
We denote by $P_\pi\in\R^{n\times n}$ the permutation matrix which permutes the columns of a matrix $\tilde A\in\R^{n\times n}$ with respect to $\pi$ when multiplying $P_\pi$ from the right to form $\tilde AP_\pi$.
\end{defn}

\begin{example}
Suppose $\pi=(234)$. The corresponding permutation matrix in $\R^{4}$ is 
	$$P_\pi=\begin{pmatrix}
		1 & 0 & 0 & 0 \\
		0 & 0 & 1 & 0 \\
		0 & 0 & 0 & 1 \\
		0 & 1 & 0 & 0
	\end{pmatrix}.
	$$
For the matrix $\tilde A=(a_1 \mid a_2 \mid a_3 \mid a_4) \in \R^{4 \times 4}$, the permutation applied from the right yields
$\tilde AP_\pi=(a_1 \mid a_4 \mid a_2 \mid a_3)$. By deleting the third row and column of $\tilde AP_\pi$, we get
	$$(\tilde AP_\pi)_3 = \begin{pmatrix}
		a_{1,1} & a_{4,1} & a_{3,1}  \\
		a_{1,2} & a_{4,2} & a_{3,2} \\
		a_{1,4} & a_{4,4} & a_{3,4}
	\end{pmatrix} = \tilde A_{3,2}\cdot\begin{pmatrix}
		1 & 0 & 0 \\
		0 & 0 & 1 \\
		0 & 1 & 0
	\end{pmatrix}
	= \tilde A_{3,\pi^{-1}(3)}\cdot P_{(23)}.
	$$
Note that although $(\tilde AP_\pi)_3$ and $\tilde A_{3,\pi^{-1}(3)}$ inherited the same columns from $\tilde A$, they have different order.
\end{example}


The next theorem is in fact stronger than the result we were looking for. It states that for any square matrix there exist at most two complex eigenvalues which are invariant under all permutation.

\begin{theorem}\label{2eig}
	Let $\tilde A\in\R^{n\times n}$.  Suppose $\lambda_0\in\C$  is an eigenvalue of $\tilde AP_\pi$ for all $\pi\in S_n$. Then,
	$$\lambda_0\in\Big\{ 0, \frac 1n \sum_{\ell=1}^{n}\sum_{m=1}^n a_{\ell,m} \Big\}.$$
\end{theorem}
\begin{proof}
	If $\lambda_0$ is an eigenvalue of $\tilde AP_\pi$ for all $\pi\in S_n$, then it is also a root of the sum of all characteristic polynomials $\chi_{\tilde AP_\pi}(\lambda)$, $\pi\in S_n$, i.e., it is a root of the polynomial 
	$$\sum_{\pi\in S_n} \chi_{\tilde AP_\pi}(\lambda) = \sum_{k=0}^n\lambda^k \sum_{\pi\in S_n}  c_k(\tilde AP_\pi).$$
	Applying equation (\ref{coeff}) yields
	\begin{align}\label{charsum}
		\sum_{k=0}^n\lambda^k\sum_{\pi\in S_n} c_k(\tilde AP_\pi) =  \lambda^n \, n!+\sum_{k=0}^{n-1}\lambda^k(-1)^{n-k} \sum_{\pi\in S_n} \sum_{\substack{I\subset[n] \\ |I|=k}} \det\rb{\rb{\tilde AP_\pi}_I}.
	\end{align}
	Let us turn our attention to the term $\det\rb{\rb{\tilde AP_\pi}_I}$. Clearly,
	$$\abs{\det\rb{\rb{\tilde AP_\pi}_I}} = \abs{\det\rb{\tilde A_{I,\pi^{-1}(I)}}},$$
	since the two matrices $\rb{\tilde AP_\pi}_I$ and $\tilde A_{I,\pi^{-1}(I)}$ only differ by a permutation of their columns. For $|I|=n-1$ the two matrices are scalars, and therefore they are identical and equal to the matrix entry $a_{\ell,\pi^{-1}(\ell)}$, where $I=[n]\setminus\{\ell\}$. For $|I|\in\{0,\ldots,n-2\}$, let $\tilde\pi_I\in S_{n-|I|}$ be the permutation such that
	\begin{equation}\label{pitilde}
		\rb{\tilde AP_\pi}_I = \tilde A_{I,\pi^{-1}(I)} P_{\tilde\pi_I}.
	\end{equation}
Hence, 
	$$\det\rb{\rb{\tilde AP_\pi}_I} = \det\rb{\tilde A_{I,\pi^{-1}(I)}}\cdot\sgn(\tilde\pi_I).$$
	Now equation (\ref{charsum}) yields
	\begin{align}\label{charcoal}
		\sum_{k=0}^n\lambda^k\sum_{\pi\in S_n} c_k(\tilde AP_\pi) &= \lambda^nn! - \lambda^{n-1} \sum_{\pi\in S_n} \sum_{\ell=1}^n a_{\ell,\pi(\ell)} \nonumber \\
		&\quad+ \sum_{k=0}^{n-2}\lambda^k(-1)^{n-k} \sum_{\pi\in S_n} \sum_{\substack{I\subset[n] \\ |I|=k}} \sgn(\tilde\pi_I)\det\rb{\tilde AP_{I,\pi^{-1}(I)}} \nonumber \\
		&= \lambda^nn! - \lambda^{n-1}  \sum_{\ell,m=1}^n a_{\ell,m}\underbrace{\abs{\cb{\pi\in S_n\mid \pi(\ell)=m}}}_{=(n-1)!} \nonumber \\
		&\quad+ \sum_{k=0}^{n-2}\lambda^k(-1)^{n-k} \sum_{\pi\in S_n} \sum_{\substack{I\subset[n] \\ |I|=k}} \sgn(\tilde\pi_I)\det\rb{\tilde AP_{I,\pi^{-1}(I)}}.
	\end{align}
	We now take a closer look at the sum
	$$\sum_{\pi\in S_n} \sum_{\substack{I\subset[n] \\ |I|=k}} \sgn(\tilde\pi_I)\det\rb{\tilde AP_{I,\pi^{-1}(I)}}.$$
	Similar to the trick that we have applied to the sum for $k=n-1$, we replace the sum over all $\pi\in S_n$ by the sum over all $J\subset[n]$ with $|J|=k$ and multiply every summand with the number of permutation $\pi\in S_n$ satisfying $\pi^{-1}(I)=J$ and $\sgn(\tilde\pi_I)=\pm 1$; i.e.
	\begin{align*}
		&\ \sum_{\pi\in S_n} \sum_{\substack{I\subset[n] \\ |I|=k}} \sgn(\tilde\pi_I)\det\rb{\tilde AP_{I,\pi^{-1}(I)}} \\
		= & \sum_{\substack{I,J\subset[n] \\ |I|,|J|=k}} \det\rb{\tilde A_{I,J}}\cdot \bigg(\abs{\cb{\pi\in S_n\mid \pi^{-1}(I)=J\land\sgn(\tilde\pi_I)=1}} \\
		&\qquad\qquad\qquad\qquad - \abs{\cb{\pi\in S_n\mid \pi^{-1}(I)=J\land\sgn(\tilde\pi_I)=-1}}\bigg) =0.
	\end{align*}
	In the last step, we exploited that the two sets of permutations are identical in their size, as they only differ by the sign of the permutation $\tilde \pi_I$.
	Inserting this result into (\ref{charcoal}) yields
	$$\sum_{k=0}^n\lambda^k\sum_{\pi\in S_n} c_k(\tilde AP_\pi) = \lambda^nn! - \lambda^{n-1}(n-1)! \, \sum_{\ell=1}^n\sum_{m=1}^n a_{\ell,m}.$$
As this polynomial has only the two roots $\lambda_1=0$ and $\lambda_2=\frac 1n\sum_{\ell=1}^n\sum_{m=1}^n a_{\ell,m}$, they are the only candidates as eigenvalues of $\tilde AP_\pi$ for all $\pi\in S_n$.
\end{proof}

\begin{remark}
	Since a singular matrix is turned into another singular matrix by applying a permutation to its columns, every singular matrix is an example for a matrix with an eigenvalue $\lambda_{1} = 0$ that is invariant under permutation of the columns. An example that uses the other option for the invariant eigenvalue
	$$\lambda_2 = \frac 1n\sum_{\ell=1}^n\sum_{m=1}^n a_{\ell,m}$$
	is the identity matrix, since $\id\cdot P_\pi=P_\pi$. Moreover, every permutation matrix has $\lambda_{2}=1$ as an eigenvalue corresponding to the eigenvector
	$$v=\frac 1{\sqrt n}\begin{pmatrix}
		1\\\vdots\\1
	\end{pmatrix}.
	$$
\end{remark}

We conclude from Theorem \ref{2eig}:

\begin{corollary}\label{eigi}
	Let $\tilde A\in\R^{n\times n}$ such that $i$ is an eigenvalue of $\tilde A$. Then there exists a permutation $\pi\in S_n$ such that $\tilde AP_\pi$ does not have eigenvalue $i$.
\end{corollary}
\begin{proof}
	The only two candidates for eigenvalues of $\tilde AP_\pi$ for all $\pi\in S_n$ are $\lambda_1=0$ and $\lambda_2=\frac 1n\sum_{\ell=1}^n\sum_{m=1}^n a_{\ell,m}$, both of which are real-valued.
\end{proof}

This Corollary answers our original question whether there exists a permutation $\pi\in S_n$ such that for given matrices $A,V\in\gl(n,\R)$, $V+iAVP_{\pi}\in\gl(n,\C)$, thus its columns forming a basis of $\C^n$. Corollary \ref{eigi} yields the existence of a permutation $\pi\in S_n$ such that for $\tilde A = V^{-1}AV$, $i$ is not an eigenvalue of $\tilde AP_\pi$. Hence, $\id+i\tilde A P_\pi\in\gl(n,\C)$ and
$$V(\id+i\tilde A P_\pi) = V + iAVP_\pi$$
is regular, as well.

We can sum up: In the case of finite-dimensional vector spaces, for all real regular matrices $\tilde A,V$ there always exists a permutation matrix $P_{\pi}$ such that $\tilde AP_\pi$ is a rebricking operator. Thus, the matrix $V +iAVP_{\pi}$ is a complex regular matrix. Its columns form a rebricked complex basis.

\section{Rebricking with Redundancy}

\subsection{Frames}
\label{sec:Frames}
Frames are powerful tools in signal and image processing which---contrary to bases---allow for redundancy. Thereby they enable a sparse signal representation and provide more robustness against noise \cite{Mallat2009}.

\begin{defn}[Frame]
	Let $X$ be a real or complex separable infinite-dimensional Hilbert space. A set $\cb{f_n: n \in\N}\subset X$ is called \emph{frame} if there exist two constants $0<c \leq C < \infty$ (the frame bounds) such that
	for all $f\in X$,
	\[c\|f\|^2 \le \sum_{n\in\N} \abs{\scp{f,f_n}}^2 \le C\|f\|^2.\]
\end{defn}

We transfer the notion of rebrickability and the rebricking operator of Definition \ref{D rebrickability Riesz bases} to frames.

\begin{defn}[Rebrickability]
\label{Def Rebrickability Frames}
 Suppose $\Hi$ is an infinite-dimensional separable real Hilbert space and $A \in \mathcal{L}(\Hi)$.
	\begin{enumerate}[label=(\roman*)]
		\item 
		Let $\{f_n:\ n \in \N\}$, $\{g_n:\ n \in \N\}$ be two frames in $\Hi$.
They are called \emph{rebrickable} if \[\{\tilde{f}_n := f_n + i g_n : n \in \N \}\] is a frame of $\Hi + i\Hi$. In this case, $\{\tilde{f}_n : n \in \N \}$ is called a complexified or rebricked frame.
		\item The operator $A$ is called \emph{rebricking operator (for frames)} if for every frame $\{f_n:\ n \in \N\}\subset \Hi$,
		\[\{\tilde{f}_n := Bf_n := f_n + i Af_n : n \in \N \}\] is a complexified or rebricked frame for $\Hi + i \Hi$.
	\end{enumerate}	
\end{defn}

\begin{remark}
	\begin{enumerate}[label=(\roman*)]
	\item Similar to the Riesz basis case, if $A$ is able to rebrick \emph{one} frame $\{f_n:\ n \in \N\}\subset \Hi$ via $\{f_n + i Af_n : n \in \N \}$, Theorem \ref{rebrickapp} will imply that it is a rebricking operator for \emph{all} real frames.
	\item In contrast to the Riesz basis case, the two approaches (i) and (ii) in Definition \ref{Def Rebrickability Frames} will turn out to be not equivalent. In fact the rebricking operator approach (ii) will be a special case of the more general frame rebricking approach (i).
	\item Note that for our frame rebricking approach we require infinite dimensional Hilbert spaces. The finite dimensional case must be treated in a different manner, as frames for the same finite dimensional Hilbert space may differ in their number of elements. There are ways to handle these issues such as adding zero-vectors to the frames with fewer elements, however these strategies are not covered by this article.
	\end{enumerate}
\end{remark}

As a general requirement for the remainder of this article we only consider infinite dimensional separable vector spaces $X$ and $\Hi$.

Before we start discussing rebricking real frames to complex ones we show that every real frame is a complex frame with the same bounds.
\begin{lemma}
	Suppose $\Hi$ is a real separable infinite dimensional Hilbert space.
	$\cb{f_n: n \in\N}\subset\Hi$ is a frame for $\Hi$ if and only if $\cb{f_n: n \in\N}\subset\Hi$ is a frame for $\Hi +i\Hi$. The frame bounds are the same.
\end{lemma}
\begin{proof}
	Suppose the real frame $\cb{f_n: n \in\N}\subset\Hi$ possesses the frame bounds $c,C > 0$.  Let $g = g_r +ig_i \in \Hi + i\Hi$ be an arbitrary element with $g_r,g_i \in \Hi$ denoting real and imaginary part. Then,
	\begin{align*}
	c\norm{g}_2^2 = c \norm{g_r}_2 ^2 + c \norm{g_i}_2^2 \leq \sum_{n\in\N} \abs{\scp{g_r,f_n}}^2 + \sum_{n\in\N} \abs{\scp{g_i,f_n}}^2 \\
	= \sum_{n\in\N} \Big(\abs{\scp{g_r,f_n}}^2 + \abs{\scp{g_i,f_n}}^2\Big)
	= \sum_{n\in\N} \abs{\scp{g_r + ig_i,f_n}}^2 = \sum_{n\in\N} \abs{\scp{g,f_n}}^2,
	\end{align*}
where we have used Pythagorean Theorem twice.
	By applying similar estimates in reversed order we get
	\begin{align*}
	\sum_{n\in\N} \abs{\scp{g,f_n}}^2 = \sum_{n\in\N} \abs{\scp{g_r,f_n}}^2 + \abs{\scp{g_i,f_n}}^2 \leq C \norm{g_r}_2 ^2 + C \norm{g_i}_2^2 = C\norm{g}_2^2.
	\end{align*}
	
	The converse is obvious.	
\end{proof}

The previous section has dealt with the class of Riesz bases and the special subclass of orthonormal bases. Both classes share the important property that all elements are equivalent up to an automorphism (a unitary automorphism in case of orthonormal bases). Furthermore every Riesz basis can be generated by applying an automorphism on an arbitrary orthonormal basis. Frames do not have these two properties.  An analogous result for frames that characterizes them in terms of operators and orthonormal bases is the following, cf. \cite[Theorem 5.5.4]{Christensen2003}.
\begin{theorem}
	\label{thm:frame_characterization}
	Let $\{e_n :\,n \in \N\} \subset X$ be any arbitrary orthonormal basis for $X$. The frames for $X$ are precisely the families $\{U e_n: \,n \in \N\}$, where $U: X\to X$ is a bounded and surjective linear operator.
\end{theorem}
Similarly we can show that frames $\cb{f_n: n \in\N}$ constitute frames $\cb{ Af_n: n\in\N}$ in $\Hi$ if and only if the linear operator $A$ is bounded and surjective. Before we prove this we state an important Lemma on surjective linear operators in the formulation of \cite[Lemma 2.4.1]{Christensen2003} which is proven in \cite[Theorem 4.13]{Rudin1991}.

\begin{lemma}\label{lem:surjective}
	Suppose $A : X \to X$ is a bounded linear operator on a Hilbert space $X$. Then the following two conditions are equivalent:
	\begin{enumerate}[label=(\roman*)]
		\item $A$ is surjective.
		\item There exists a constant $C > 0$ such that for all $y \in X$ 
		$$\norm{A^*y} \geq C\norm{y}.$$
	\end{enumerate}
\end{lemma}

\begin{corollary}\label{frameb}
	Suppose $X$ is a real or complex separable Hilbert space. Let $A\in\mathcal{L}(X)$. Then, the following three statements are equivalent:
	\begin{enumerate}[label=(\roman*)]
		\item For all frames $\cb{f_n: n \in\N}$ in $X$, the family $\cb{ Af_n: n\in\N}$ is a frame in $X$.
		\item There exists a frame  $\cb{f_n: n \in\N}$ in $X$ such that the family $\cb{ Af_n: n\in\N}$ is a frame in $X$.
		\item $A$ is surjective.
	\end{enumerate}
\end{corollary}

\begin{proof}
	(i) $\Rightarrow$ (ii):
	
	This implication is trivial.
	
	(ii) $\Rightarrow$ (iii):
	
	Suppose $0<c<C$ are the frame bounds of the frame $\cb{ f_n: n\in\N}$.
 	By (ii), $\cb{ Af_n: n\in\N}$ is also a frame, hence there exist frame bounds $0<c'<C'$ such for all $f\in X$:
	\[c'\|f\|^2 \le \sum_{n\in\N} \abs{\scp{f, Af_n}}^2 \le C'\|f\|^2.\]
	We can transform the middle term into 
	\begin{equation}\label{transpose}
	\sum_{n\in\N} \abs{\scp{f, Af_n}}^2 = \sum_{n\in\N} \abs{\scp{ A^*f,f_n}}^2.
	\end{equation}
	Using the frame bounds $0<c<C$ of the frame $\cb{f_n:n\in\N}$ in $X$, we get
	\[c\| A^*f\|^2\le\sum_{n\in\N} \abs{\scp{f, Af_n}}^2\le C\| A^*f\|^2\]
	for all $f\in X$. Consequently,
	\[c\| A^*f\|^2\le C'\|f\|^2\quad\text{and}\quad c'\|f\|^2\le C\| A^*f\|^2.\]
	The first inequality guarantees the boundedness of $A^*$ and hence of $A$, while the second inequality ensures that $A$ is surjective due to Lemma \ref{lem:surjective}.
	
	(iii) $\Rightarrow$ (i):
	
	This is \cite[Corollary 5.3.2]{Christensen2003}. 
\end{proof}

We see: The natural homomorphisms on the set of frames of a Hilbert space are the bounded surjective linear operators. Unfortunately, there are pairs of frames, which cannot be mapped onto each other by such an operator $A: X \to X$.

\begin{example}[Incompatible frames]
	\label{ex:incomp_frames}
	Consider an arbitrary orthonormal basis $\{e_n : n \in \N\} \subset X$. This basis can be transformed into two frames by defining the frames as $$f_1 := e_1, \quad f_2 := e_1, \quad f_{i+1} := e_i \quad \mbox{for all } i \geq 2$$ and 
	$$g_1 := e_1, \quad g_2 := e_2, \quad g_{i+1} := e_i \quad \mbox{for all } i \geq 2.$$
	Clearly it is impossible to define an element-wise mapping of the first frame onto the second one or vice-versa.
\end{example}

Hence the structure of the set of all frames for a Hilbert space $X$ is much more complicated than the set of all orthonormal bases or Riesz bases due to the lack of isomorphisms between its elements.
Frames may differ by their amount of excess (redundancy). If we have a  bounded, surjective linear operator $A: X \to X$  which maps a frame onto another, then there are two possible cases: Either the excess of the two frames is equal, or the operator $A$ increases the excess, i.e. there are chains of frames linked by bounded surjective linear operators generating increased excess. What complicates the situation even more: There is not a single chain but there are multiple chains having incompatible excess as shown in the Example \ref{ex:incomp_frames} above. Therefore the set of all frames is a tree whose root is the set of all Riesz bases. We investigate this further in subsection \ref{subsec:rebricking_incomp_frames}.

In summary the canonical generalization of our ansatz $\id +iA$ to the frame world is to require $A$ to be bounded and surjective. But as seen in the example above it is not possible to cover all couples of frames:

\begin{example}	
	\label{ex:nonrebrickable_frames}
	We again use the two frames $\{f_n :\,n \in \N\},\{g_n :\,n \in \N\} \subset \Hi$ from Example \ref{ex:incomp_frames}. They cannot be mapped onto each other, i.e. there is no linear operator $A$ to use either of them as the real part and the other one as the imaginary part. 
	Even though the system $\{f_n + i g_n : n \in \N \}$ is a frame. For all $f\in\Hi + i \Hi$,
	
	\begin{align*}\sum_{n \in \N} \abs{\langle f,f_n+ig_n\rangle}^2 &= \sum_{n \in \N} \abs{\langle f,e_n+ie_n\rangle}^2 + \abs{\langle f,e_1+ie_2\rangle}^2 \\
	&= \sum_{n \in \N} \abs{\langle (1-i)f,e_n\rangle}^2 + \abs{\langle f,e_1+ie_2\rangle}^2 \\
	&= \norm{(1-i)f}^2 + \abs{\langle f,e_1+ie_2\rangle}^2\\
	&= 2\norm{f}^2 + \abs{\langle f,e_1+ie_2\rangle}^2. \end{align*}
	Clearly this expression can be bounded from above by $4\norm{f}^2$ because \begin{align*}2\norm{f}^2 + \abs{\langle f,e_1+ie_2\rangle}^2 &\leq 2\norm{f}^2 + \norm{f}^2\norm{e_1+ie_2}^2\\ &= 2\norm{f}^2 + \norm{f}^2(\norm{e_1}^2+\norm{e_2}^2) = 4\norm{f}^2.\end{align*}
	and from below by $2\norm{f}^2$. Thus we have \[2\norm{f}^2 \leq \sum_{n \in \N} \abs{\langle f,f_n+ig_n\rangle}^2 \leq 4\norm{f}^2.\] 
	In the notion of Definition \ref{Def Rebrickability Frames}, the two frames are rebrickable, however there doesn't exist a rebricking operator $A$ such that $f_{n} + i g_{n} = (\id + iA) f_{n}$ for all $n\in\N$.
\end{example}

For frames, we have two possible approaches for rebricking:
\begin{enumerate}[label=(\roman*)]
\item via rebricking two real frames and 
\item via a rebricking operator $A$ applied to arbitrary real frames.
\end{enumerate}

Note that the ansatz (ii) via the rebricking operator does not translate 1-1 from the Riesz basis setting to the frame setting, as it only allows for a subset of frame couples to be investigated, namely the couples  $\{ (f_n, \, g_n := Af_n) : \, n \in \N\}$ where $\{f_n :\,n \in \N\} \subset \Hi$ is a frame and $A : \Hi \to \Hi$ is bounded, linear and surjective. Intuitively spoken, $\{g_n :\,n \in \N\}$ has at least the same excess as $\{f_n :\,n \in \N\}$. The condition $g_n = Af_n$ for all $n\in\N$ implies that the two frames are compatible via the mapping $A$, i.e. they are within the same chain in the tree of frames.
 
Similar to the cases of Riesz bases and orthonormal bases not all couples $\{ (f_n, \ g_n := Af_n) :\, n \in \N\}$ constitute a frame $\{ f_n +i g_n ,\, n \in \N\} \subset \Hi + i\Hi$. There need to be conditions on $A$.
Luckily these conditions are similar to the condition on the spectrum of the operator $A$ in the Riesz basis case. In fact the condition for frames refers to on the approximate point spectrum \cite{Rudin1991,Werner2011,mueller2000} and the surjectivity spectrum \cite{Laursen1989} which are both subsets of the classical spectrum .

\begin{defn}[Approximate point spectrum] 
	Let $A:X \to X$ be a bounded linear operator on an arbitrary (real or complex) Hilbert space $X$. $\lambda \in \C$ is called approximate eigenvalue if there exists  a sequence $\{f_n : n\in\N\} \subset X$ such that $\norm{f_n} = 1$ for all $n \in \N$ and \[\lim_{n\to\infty}\norm{Af_n - \lambda f_n} = 0.\] The set of approximate eigenvalues is known as the approximate point spectrum, denoted by $\sigma_\mathrm{app}(A)$.
\end{defn}

\begin{defn}[Surjectivity spectrum]
	Let $A:X \to X$ be a linear operator on an arbitrary (real or complex) Hilbert space $X$. The surjectivity spectrum of $A$ is defined as
	$$\sigma_\mathrm{sur}(A) = \cb{\lambda\in\C:\ (A-\lambda\,\id)X=X}.$$
\end{defn}

For a more detailed survey on the surjective spectrum, we refer the reader to \cite{Laursen1989}.

As the next proposition shows, the approximate point spectrum and the surjectivity spectrum have an innate relation.

\begin{proposition}\label{prop:sigma}
	Let $A:X \to X$ be a linear operator on an arbitrary (real or complex) Hilbert space $X$. Then,
	$$\sigma_\mathrm{sur}(A)=\overline{\sigma_\mathrm{app}(A^*)}$$
\end{proposition}
For a proof see \cite{Laursen2000} up to complex conjugation.

\begin{theorem}\label{rebrickapp}
	Let $\Hi$ be a real separable Hilbert space and let $A\in \cal{L}(\Hi)$. Then, the following three statements are equivalent:
	\begin{enumerate}[label=(\roman*)]
		\item For all frames $\cb{f_n: n \in\N}$ in $\Hi$, the family $\cb{f_n+iAf_n: n\in\N}$ is a frame in $\Hi+i\Hi$.
		\item There exists a frame $\cb{f_n: n \in\N}$ in $\Hi$ such that the family $\cb{f_n+iAf_n: n\in\N}$ is a frame in $\Hi+i\Hi$.
		\item $i\notin\sigma_\mathrm{sur}(A)$.
		\item $-i \notin \sigma_\mathrm{app}(A^*)$.
	\end{enumerate}
\end{theorem}

\begin{proof}
	The equivalences (i) $\gdw$ (ii) $\gdw$ (iii) follow immediately from Corollary \ref{frameb}.
	
Moreover, we conclude the equivalence (iii) $\Leftrightarrow$ (iv) via Proposition \ref{prop:sigma}.
\end{proof}

We summarize the results: 
\begin{corollary}
\label{Kor rebrickability frames}
		Let $\Hi$ be a real separable Hilbert space, let $\{f_n: n\in\N\}$ be a frame in $\Hi$ and let $A\in \mathcal{L}(\Hi)$. Then, the following two statements are equivalent:
		\begin{enumerate}[label=(\roman*)]
			\item The families $\cb{Af_n: n\in\N}$ and $\cb{f_n+iAf_n: n\in\N}$ are frames in $\Hi+i\Hi$.
			\item $A$ is surjective and $-i \notin \sigma_\mathrm{app}(A^*)$.
		\end{enumerate}
\end{corollary}

\begin{remark}\label{rb vs frame}
We compare  Corollary \ref{Kor rebrickability frames} with the special case of the Riesz bases in Theorem \ref{thm_lift_riesz_bases}.
	\begin{enumerate}[label=(\roman*)]
		\item In the Riesz basis case, rebricking operators and rebricking pairs of Riesz bases are equivalent concepts. In the frame case, this equivalence does not hold anymore. Rebrickability by a rebricking operator is a special case of the general rebrickability of pairs of frames.
		\item In the Riesz basis case, a rebricking operator $A$ has to be invertible. For frames, this property is relaxed to only requiring that the rebricking operator $A$ is surjective. 
		\item For the rebrickability relation in the Riesz basis case, we had symmetry, cf. Lemma \ref{L Riesz case Properties of the rebrickability relation}.
		For the frame case this in general does not hold anymore: The imaginary part is built from  the surjective operator $A$ applied to the real part. Its excess is at least the one of the real part. If the excess is increased with the rebricking operator A then an inverse operator to map the imaginary part to the real part does not exist.
		\item In the Riesz basis case the rebricking operator $A$ must not contain $-i$ in its spectrum. This is generalized to the condition that $-i$ must not be in the approximate spectrum. This is indeed a generalization as $$\sigma_p(A^*) \subset \sigma_\mathrm{app}(A^*) \subset \sigma(A^*)$$ where $\sigma_p(A^*)$ denotes the point spectrum (set of eigenvalues).
	\end{enumerate}

	The second observation shows once again the interesting differences between the finite-dimensional world and the infinite-dimensional world. In the case of finite-dimensional vector spaces we have $\sigma_p(A^*) = \sigma(A^*)$, and here, the approximate spectrum is no concept of its own. Moreover, on finite-dimensional vector spaces, there do not exist surjective linear operators which are not injective.
\end{remark}

From these thoughts we can immediately deduce necessary and sufficient criteria for rebricking.
\begin{corollary}
\label{Cor necessary and not sufficient}
	Let $A \in \mathcal{L}(\Hi)$.
	\begin{enumerate}[label=(\roman*)]
		\item The property that $A^*$ has no eigenvalue $-i$ is necessary but not sufficient for $\id +iA$ to be surjective.
		\item If $-i \notin \sigma(A^*)$ then $\id +iA$ is surjective. (This is a sufficient but not necessary condition)
	\end{enumerate}	
\end{corollary}

The next lemma characterizes conditions for $A$ being suitable for frame rebricking. There the complex conditions (i) and (ii) can be converted to the real condition (iii). Furthermore we can replace the often difficult way of proving that a concrete operator $\id + iA$ is surjective by an inequality condition on the adjoint operator. It turns out that the latter is extremely useful in lots of cases.
\begin{lemma}
	\label{frame-rebricking-conditions}
	Let $A \in \mathcal{L}(\Hi)$ be surjective. Then there are equivalent: \begin{enumerate}[label=(\roman*)]
		\item $\id + iA$ is surjective in $\Hi + i\Hi$;
		\item $\id - iA$ is surjective in $\Hi + i\Hi$;
		\item $\id + A^2$ is surjective in $\Hi$.
	\end{enumerate}
	
	If in addition for $A^*$ there exists $c > 1$ such that for all $y \in \Hi$ 
	$$ \norm{A^*y} \geq c\norm{y}$$
	 then $\id + iA$ is surjective onto $\Hi + i\Hi$. 
	 
	 The converse is not true in general.
\end{lemma}

\begin{proof}
	The three items are identical to the corresponding Lemma \ref{lem:complex-bijective-to-real-bojective} for Riesz bases.
	
	To prove the supplement we note that
	\[\norm{(\id - iA^*)y} \geq \abs{\norm{A^*y} - \norm{y}} \geq c\norm{y} - \norm{y} = (c-1)\norm{y}. \] 
	Thus by \cite[Corollary 5.3.2]{Christensen2003} we conclude that $\id +iA$ is surjective.
	
	As a counterexample for the converse simply consider $A = \id$.
\end{proof}

\begin{example}
	We construct an operator $A \in \mathcal{L}(\ell^2(\N;\R))$ such that
	\begin{enumerate}[label=(\roman*)]
		\item $A$ is real and surjective, but not injective.
		\item $\id +iA$ is surjective, but not injective.
		\item $A^*$ has no eigenvalue $-i$.
	\end{enumerate} 
	
	We take \[A: \ell^2(\N;\R) \to \ell^2(\N;\R),\quad(Ax)_n := 2x_{n+1}.\] Clearly
	\begin{itemize}
		\item $A$ is not injective, as $\ker A = \Span\{(1,0,0,\ldots)\}$.
		\item $A$ is surjective.
		\item The set of eigenvalues of $A$ is $\{\lambda \in \C: \abs{\lambda} < 2 \}$, thus $\id +iA$ is not injective.
		\item $\id + iA$ is surjective, because $\norm{A^*y} = 2\norm{y}$ for all $y \in \ell^2(\N;\C)$. In particular $A^*$ has no eigenvalue $-i$.
	\end{itemize}
\end{example}

The next two examples illustrate Corollary \ref{Cor necessary and not sufficient}:

\begin{example}[Adjoint Operator without eigenvalue $-i$ not sufficient]
	\label{example-eigenvalue-not-sufficient}
	We construct a bijective operator $A \in \mathcal{L}(\RLtwo)$ (in this case $A^*$ is bijective, too) such that $A^*$ has no eigenvalue $-i$ but an approximate eigenvalue $-i$, i.e. $\id +iA$ is not surjective. 
	
	For this it suffices to construct $m \in L^\infty(\R,\C)$ such that 
	\begin{itemize}
	\item $m(\omega) = \overline{m(-\omega)}$, $m(\omega) \neq 0$, $m(\omega) \neq i$ almost everywhere, and
	\item   $m$ has a limit point $i$.
	\end{itemize}
	For instance, this can be done by setting 
	$$m|_{[2,\infty)}(\omega) = i-\frac{i}{\omega} \quad \mbox{and}\quad m|_{(0,2)} = \frac{i}{2},$$ and extending the function to $\R$. Then we define the operator $A$ as \[A: \RLtwo \to \RLtwo,\quad Af := \F^{-1}(m\, \F f),\]
	where $\F$ denotes the Fourier transform.
	Because of the property $m(\omega) = \overline{m(-\omega)}$ the operator is real-valued, hence well-defined. By elementary computations we can see that the adjoint operator is \[A^*f = \F^{-1}(\overline{m}\, \F f).\]
	\begin{itemize}
		\item $A$ is bijective because $1/m \in L^\infty(\R)$ as $\frac{1}{2} \leq \abs{m(\omega)} \leq 1$ almost everywhere.
		\item $\id - iA^*$ is injective, i.e. $A^*$ has no eigenvalue $-i$ because $$-if = A^*f \quad \Leftrightarrow\quad  -i\F f = \overline{m} \F f$$ for some $f \in \CLtwo$ has  the unique solution $f = 0$ as $\overline{m} \neq -i$ almost everyhwere. This also shows that $\id +iA$ has dense range.
		\item $\id +iA$ is not surjective. To see this consider the sequence $\{g_n : n \in \N\} \subset \CLtwo $ with  $\F g_{n} = \chi_{[n,n+1)}$. 
		
		\[\norm{(\id -iA^*)g_{n}} = \norm{\F^{-1}((1-i\,\overline{m})\F g_{n})} = \norm{(1-i\,\overline{m})\F g_{n}}.\] There exists no constant $c > 0$ such that for all $n\in\N$ \[\norm{(1-i\,\overline{m})\F g_{n}} \geq c\norm{g_{n}},\] 
because $\norm{(1-i\,\overline{m})\F g_{n}} \rightarrow 0$ for $n \rightarrow \infty$ by construction, but $\norm{g_n}$ is constant for all $n \in \N$.
	\end{itemize}
\end{example}

\subsection{Rebricking of incompatible frames}
\label{subsec:rebricking_incomp_frames}
As mentioned earlier, Theorem \ref{rebrickapp} has the slight drawback that two frames which are to be rebricked together have to be compatible in the sense that there exists a surjective operator mapping one frame onto the other. To overcome this issue and thereby extend the family of rebricking operators, we first examine the structure of frames more closely by introducing and analyzing a partial order in the set of frames.

\begin{defn}
Let $X$ be a separable, real or complex Hilbert space. 
We denote 
\begin{enumerate}[label=(\roman*)]
	\item the set of surjective, linear, and continuous operators mapping $X$ onto $X$ by $\sur(X)$, and
	\item the set of frames on $X$ by $\fr(X)$.
\end{enumerate}
We introduce a partial order on $\sur(X)$ by defining for $A,B\in\sur(X)$ that
$$A\le B \quad\Leftrightarrow\quad \exists\ T\in\sur(X):\ A=TB.$$
Similarly, we introduce a partial order on $\fr(X)$ by defining for two frames $F=\cb{f_n:n\in\N}$, $G=\cb{g_n:n\in\N}\in\fr(X)$ that
$$F\le G \quad\Leftrightarrow\quad \exists\ T\in\sur(X):\ F=TG.$$
If $A\le B$, we say that $A$ is smaller than $B$ and similarly for frames.

The two partial orders naturally allow for two corresponding equivalence relations by defining for $A,B\in \sur(X)$ that
$$A\sim B \quad\Leftrightarrow\quad (A\le B)\land (B\le A)$$
and for $F,G\in\fr(X)$ that
$$F\sim G\quad\Leftrightarrow\quad (F\le G)\land (G\le F).$$
\end{defn}

Our next goal is to analyze the structure of the \FH{tree} that arises from this partial order and especially to identify the equivalence classes.

For the partial order on $\sur(X)$, a natural starting point is to consider whether the respective operators are injective, since the equivalence class of largest operators with respect to the partial order is the set of bijective operators. To this end, we use the dimension of an operator's kernel as a measure for its non-injectivity. The next lemma shows a basic property of the kernel dimension on $\sur(X)$.

\begin{lemma}\label{dimker}
Let $A\in\mathcal{L}(X)$, $B\in\sur(X)$. Then,
$$\ker(AB) = \ker B \oplus \rb{B\big|_{\rb{\ker B}^\perp}}^{-1} \ker A$$
and
$$\dim(\ker(AB))=\dim(\ker A)+\dim(\ker B),$$
where $\rb{\ker B}^\perp$ denotes the orthogonal complement of $\ker B$ within $X$.
\end{lemma}
\begin{proof}
Clearly, $\ker(B) \subset \ker(AB)$. To show the first of the two equations, we prove that
$$\rb{\ker B}^\perp\cap\ker(AB)=\rb{B\big|_{\rb{\ker B}^\perp}}^{-1} \ker A.$$
As $B$ is surjective,
$$B\rb{\rb{\ker B}^\perp\cap\ker(AB)}=\cb{Bx\ |\ x\in \rb{\ker B}^\perp\land ABx=0}=\ker A.$$
A quick examination reveals that the restricted operator 
$$\tilde B :=B\big|_{\rb{\ker B}^\perp}:\rb{\ker B}^\perp\to X$$ is bijective, as $\ker\big(\tilde B\big)=\{0\}$ and $B\rb{\rb{\ker B}^\perp}=X$. Therefore, we obtain
$$\rb{\ker B}^\perp\cap\ker(AB)=\tilde B^{-1} \ker A.$$
Since $\tilde B^{-1} \ker A\subset \rb{\ker B}^\perp$, we can conclude that $\ker B\cap\tilde B^{-1} \ker A=\{0\}$. So, their dimensions simply add up, i.e.
\begin{align*}
\dim(\ker(AB)) &=\dim(\ker B)+\dim\rb{\tilde B^{-1}\ker A}\\
&=\dim(\ker A)+\dim(\ker B).
\end{align*}


\end{proof}

This immediately sheds some light on the structure of the two partial orders for surjective operators and frames, respectively.

\begin{proposition}\label{bijeq}
Let $A,B\in\sur(X)$. Then,
$$A\sim B \quad\Leftrightarrow\quad \exists\ T:X\to X\text{ bijective }: A=TB.$$
Similarly, for $F,G\in\fr(X)$. Then,
$$F\sim G \quad\Leftrightarrow\quad \exists\ T:X\to X\text{ bijective }: F=TG.$$
\end{proposition}
\begin{proof} We first prove the result for $\sur(X)$. \\
``$\Rightarrow$'': If $A\sim B$, there exist $S,T\in\sur(X)$ such that
$$A=TB\quad\text{and}\quad B=SA.$$
By Lemma \ref{dimker}, we obtain
$$\dim(\ker A) = \dim(\ker T)+\dim(\ker B)\quad\text{and}\quad \dim(\ker B) = \dim(\ker S)+\dim(\ker A).$$
Therefore, $\dim(\ker S) = \dim(\ker T) = 0$ and so $A=TB$ for $T$ bijective.

``$\Leftarrow$'': If $A=TB$, we have by definition that $A\le B$. Moreover, $B=T^{-1}A$ and therefore $B\le A$. Consequently, $A\sim B$.

The result for frames can be immediately deduced from the result for surjective operators.
\end{proof}

In order to show how the kernel of a surjective operator translates into the set of frames, we need to introduce new terminology.

\begin{defn}[Kernel of a frame]
Let $F=\cb{f_n: n\in\N}$ be a frame. Then,
$$\ker(F) := \cb{c\in\ell^2(\N): \sum_{n\in\N}c_nf_n=0}.$$
\end{defn}

Note that this is a short notation for the kernel of the frame's synthesis operator.

The next lemma now shows a connection between the kernel of a surjective operator and a frame.

\begin{lemma}\label{dimker2}
Let $A\in\sur(X)$ and let $R=\cb{r_n:n\in\N}$ be a Riesz basis of $X$. Then, $\ker A$ is isomorphic to $\ker(AR)$. In particular this implies that
$$\dim(\ker A)=\dim(\ker(AR)).$$
\end{lemma}
\begin{proof}
Since $A$ is a surjective operator, $F:=AR=\cb{f_n:n\in\N}$ constitutes a frame. For every $x\in\ker A$ there exists a unique representation 
$$x=\sum_{n\in\N} c_nr_n,$$
as $R$ is a Riesz basis. Since $A$ is linear and continuous, we get for $x\in\ker A$ that
$$0=Ax = \sum_{n\in\N}c_nAr_n = \sum_{n\in\N}c_nf_n.$$
This relation works in the other direction, as well. If $\sum_{n\in\N}c_nf_n=0$, we obtain
$$0=\sum_{n\in\N}c_nf_n=\sum_{n\in\N}c_nAr_n=A\sum_{n\in\N}c_n r_n.$$
Therefore, $\sum_{n\in\N}c_nf_n\in\ker A$. Hence, $\ker A$ is isomorphic to $\ker F$ and so
$$\dim(\ker A) = \dim(\ker F).$$
\end{proof}

The following theorem now shows an equivalent characterization of the partial order between two surjective operators or frames, respectively.

\begin{theorem}\label{partord}
Let $A,B\in\sur(X)$. Then,
$$A\le B\quad\Leftrightarrow\quad \ker A \supset \ker B.$$
Similarly, for $F,G\in\fr(X)$,
$$F\le G\quad\Leftrightarrow\quad \ker F \supset \ker G.$$
\end{theorem}
\begin{proof} We prove the result for surjective operators. \\
``$\Rightarrow$'': By definition $A\le B$, if and only if there exists $T\in\sur(X)$ such that $A=TB$. Therefore,
$$\ker A = \ker(TB) = \cb{x\in X\mid TBx=0} \supset \cb{x\in X\mid Bx=0} = \ker B.$$
``$\Leftarrow$'': Let $V:=\ker B$. As in the proof of Lemma \ref{dimker}, the operator
$$\tilde B:= B\big|_{\rb{\ker B}^\perp}=B\big|_{V^\perp}$$
is bijective. We prove $A\le B$ by constructing an operator $T\in\sur(X)$ such that $A=TB$. Our choice is $T=A\tilde B^{-1}$ which is surjective, as $\tilde B^{-1}:X\to(\ker B)^\perp$ and $A:(\ker B)^\perp\to X$ are surjective. Due to $TB=A\tilde B^{-1}B$, we first examine the operator $\tilde B^{-1}B$. To this end we make use of the unique orthogonal decomposition $x = x_V+x_{V^\perp}$ for every $x \in X$ where $x_V \in V$ and $x_{V^\perp} \in V^\perp$. We obtain that for all $x\in X$,
$$\tilde B^{-1} B x = \tilde B^{-1} B \rb{x_V+x_{V^\perp}} = \tilde B^{-1} Bx_{V^\perp} = \tilde B^{-1} B\big|_{V^\perp} x_{V^\perp} = x_{V^\perp}.$$
Therefore, $\tilde B^{-1}B=P_{V^\perp}$, where $P_{V^\perp}$ is the orthogonal projector onto $V^\perp$. Now, for all $x\in X$,
$$A\tilde B^{-1} B x = A P_{V^\perp}x= A x_{V^\perp}.$$
As $V=\ker B\subset\ker A$, we get that $x_V\in\ker A$. Therefore, for all $x\in X$,
$$A\tilde B^{-1} B x = A x_{V^\perp} = A\rb{x_{V^\perp}+x_V}=Ax.$$
Hence, $A=TB$ and $A\le B$. \\[3mm]
The result for frames can be immediately deduced from the result for surjective operators.
\end{proof}

As a corollary of this result we obtain a characterization for two surjective operators or frame, respectively, belonging to the same equivalence class.

\begin{corollary}\label{equivcl}
Let $A,B\in\sur(X)$. Then,
$$A\sim B\quad\Leftrightarrow\quad \ker A = \ker B.$$
Similarly, for $F,G\in\fr(X)$,
$$F\sim G\quad\Leftrightarrow\quad \ker F = \ker G.$$
\end{corollary}
\begin{proof} Direct application of Theorem \ref{partord}.
\end{proof}

For convenience, we now introduce a notation for spaces of surjective operators and frames respectively, with a certain kernel dimension.

\begin{defn}
For $k\in\N_0\cup\{\infty\}$, let 
$$\sur_k(X) := \cb{A\in\sur (X):\ \dim(\ker(A))=k}$$
and
$$\fr_k(X) := \cb{F\in\fr (X):\ \dim(\ker(F))=k}.$$
\end{defn}

To give examples for such sets, $\sur_0(X)$ is the set of all bijective, linear, bounded operators mapping $X$ onto $X$ and $\fr_0(x)$ constitutes the set of all Riesz bases in $X$.

We will now use the previous results on the inherent structure of the set of frames to prove a new result on the rebrickability of frames. To this end, we first show that every surjective operator can be transferred into a surjective operator of higher or equal kernel dimension by multiplying with a surjective operator from the right hand side.

\begin{lemma}\label{surtrans}
Let $k\in\N_0\cup\{\infty\}$, $\ell\in\N_0$, $k\ge\ell$, such that $A\in\sur_k(X)$, $B\in\sur_\ell(X)$. Then there exists $T\in \sur_{k-\ell}(X)$ such that $A=BT$.
\end{lemma}
\begin{proof}
We first prove this result for $k=\ell$. 

As in the proof of Lemma \ref{dimker}, the operator
$$\tilde B:= B\big|_{\rb{\ker B}^\perp}=B\big|_{V^\perp}$$
is bijective. Let $S$ be a unitary operator such that
$$S(\ker A)=\ker B$$ 
and let $P_{\ker A}$ denote the orthogonal projection onto $\ker A$. We define 
$$T := \tilde B^{-1}A + SP_{\ker A}$$ 
and prove that this operator is bijective and satisfies $A=BT$.

For the bijectivity of $T$: We can decompose each $x\in X$ into a component in $\ker A$ and a component in $\rb{\ker A}^\perp$, i.e.
$$x=x_{\ker A}+x_{\rb{\ker A}^\perp}.$$
Applying $T$ to $x$, we obtain
$$Tx=\tilde B^{-1}Ax+SP_{\ker A}x=\tilde B^{-1}Ax_{\rb{\ker A}^\perp}+SP_{\ker A}x_{\ker A}.$$
Now let $y\in X$ be arbitrary. We can decompose $y$ into a component in $\ker B$ and a component in $\rb{\ker B}^\perp$, i.e.
$$y=y_{\ker B}+y_{\rb{\ker B}^\perp}.$$

Since $\tilde B^{-1}A:X\to\rb{\ker B}^\perp$ and $SP_{\ker A}:X\to\ker B$, the equation $Tx=y$ can be split into
\begin{align*}
y_{\ker B} &= SP_{\ker A}x_{\ker A}, \\
y_{\rb{\ker B}^\perp} &= \tilde B^{-1} Ax_{\rb{\ker A}^\perp}.
\end{align*}
As $\tilde B^{-1}A:\rb{\ker A}^\perp\to\rb{\ker B}^\perp$ and $SP_{\ker A}:\ker A \to \ker B$ are bijective, both upper equations can be solved uniquely with $x_{\ker A}\in\ker A$ and $x_{\rb{\ker A}^\perp}\in \rb{\ker A}^\perp$ for every $y\in X$. Hence, $T$ is bijective.

We show $A=BT$: We have
$$BT = \underbrace{B\tilde B^{-1}}_{\normalsize{=\id}}A + BSP_{\ker A} = A + BSP_{\ker A}.$$
$BSP_{\ker A}=0$, since $\ran\rb{SP_{\ker A}}=\ker B$. Therefore, $BT=A$.

Now let $k\ge \ell$. Due to Lemma \ref{dimker}, for every $C\in\sur_{k-\ell}(X)$, $BC\in\sur_k(X)$. Now we can apply the above result for $k=\ell$ to obtain that there exists a bijective operator $\tilde T\in \mathcal{L}(X)$ such that $A=BC\tilde T$. Choosing $T=C\tilde T$ finishes the proof.
\end{proof}

\begin{lemma}\label{opprod}
Let $A\in\sur(X)$ and let $T\in\mathcal{L}(\Hi)$. Then, the following two statements are equivalent:
\begin{enumerate}[label=(\roman*)]
\item $A T \in\sur (X)$.
\item $\ran T+\ker A= X$.
\end{enumerate}

\end{lemma}
\begin{proof}
(i) $\Rightarrow$ (ii): \\[3mm]
Since $A$ and $AT$ are surjective, we obtain that for every $y\in X$, there exists $x\in X$ such that $ATx=Ay$. Therefore, $y\in\ran T+\ker A$. Since this holds for all $y\in X$, (ii) follows.
\\[3mm]
(ii) $\Rightarrow$ (i): \\[3mm]
$AT\in\sur(X)$, since
$$\ran(AT) = A\ran T = A\rb{\ran T + \ker A} = A(X) = X.$$
\end{proof}

\begin{theorem}\label{frrebrick}
Let $\Hi$ be a real, infinite-dimensional, separable Hilbert space and let $A,B,S\in\sur(X)$ such that $B=AS$. Then, the following two statements are equivalent:
\begin{enumerate}[label=(\roman*)]
\item $A + i B = A (\id + iS) \in\sur (\Hi+i\Hi)$.
\item $\ran(\id+iS)+\rb{\ker A+i\ker A}= \Hi+i\Hi$.
\end{enumerate}

\end{theorem}
\begin{proof}
Apply Lemma \ref{opprod} with $T=\id+iS$.
\end{proof}

In the previous theorem, $S$ is the rebricking operator and from $(ii)$ it immediately becomes apparent that $\id+i S$ does not have to be surjective for $S$ to function as a rebricking operator. This marks the main difference between Theorem \ref{rebrickapp} and Theorem \ref{frrebrick}.

To illustrate this, we will examine the following example.

\begin{example}
Let $S\in\mathcal{L}\rb{\ell^2(\N;\R)}$ with the matrix representation with respect to the canonical basis 
$$S=\begin{pmatrix}
0 & 1 & 0 & 0 & 0 & \cdots \\
-1 & 0 & 0 & 0 & 0 & \cdots \\
0 & 0 & 1 & 0 & 0 & \cdots \\
0 & 0 & 0 & 1 & 0 & \cdots\\
0 & 0 & 0 & 0 & 1 & \\
\vdots & \vdots & \vdots & \vdots &  & \ddots \\
\end{pmatrix}.
$$
This operator acts as a 90 degree rotation on the $e_1$-$e_2$-plane.
Clearly, $S\in\sur\rb{\ell^2(\N;\R)}$. Moreover, $\id+iS\notin\sur\rb{\ell^2(\N;\C)}$ as
$$\id+iS=\begin{pmatrix}
1 & i & 0 & 0 & 0 & \cdots \\
-i & 1 & 0 & 0 & 0 & \cdots \\
0 & 0 & 1+i & 0 & 0 & \cdots \\
0 & 0 & 0 & 1+i & 0 & \cdots\\
0 & 0 & 0 & 0 & 1+i & \\
\vdots & \vdots & \vdots & \vdots &  & \ddots \\
\end{pmatrix}
$$
and e.g. $e_1,e_2\notin\ran(\id+iS)$. Therefore, $S$ cannot be used for rebricking in the sense of Definition \ref{Def Rebrickability Frames}. However, due to Theorem \ref{frrebrick} it is sufficient to find a surjective operator $A\in\mathcal{L}\rb{\ell^2(\N;\R)}$ with \[\ker(A)\supset\Span\{e_1,e_2\},\] since \[\Span\{e_3,e_4,e_5,\ldots\}\subset\ran(\id+iS).\] A well-known example for a surjective operator with a nontrivial kernel is the left shift operator \[L\in\mathcal{L}\rb{\ell^2(\N;\R)},\, (Lx)_n=x_{n+1}.\] Its kernel is $\ker(L)=\Span\{e_1\}$. This gives us an example $A=L^2$ which is both surjective and satisfies $\ker(A)=\Span\{e_1,e_2\}$. We can use $A$ and $S$ for rebricking, as for all $x\in\ell^2(\N;\C)$
\begin{align*}
A(\id+iS)x &= A(x_1+ix_2,-ix_1+x_2,(1+i)x_3,(1+i)x_4,\ldots) \\
&= ((1+i)x_3,(1+i)x_4,\ldots).
\end{align*}
Therefore, $A(\id+iS)\in\sur(\ell^2(\N;\C))$. Note that $A = L$ would also work for rebricking with the given operator $S$.

\end{example}

\subsection{Parseval Frames}
The Parseval frames are the convenient class of frames where the frame operator is the identity. Therefore, 
the inversion of the frame operator is trivial and frame decomposition series has a simple form. This is exploited in  applications, see e.g. \cite{GittaShearlets,Balazs2017}

\begin{defn}[Parseval frame]
	Let $X$ be a real or complex separable Hilbert space. A frame $\cb{\phi_n: n \in\N}\subset X$ is called \emph{Parseval frame} if its upper and lower frame bounds $c$ and $C$ are both equal to 1.
\end{defn}

As first step towards a rebricking theory for Parseval frames we characterize Parseval frames $\{\phi_n: n \in \N\} \subset X$ in terms of operators $U : X \to X$ and orthonormal bases $\{e_n: n \in \N\} \subset X$ for $X$ such that $\phi_n = Ue_n$ for all $n \in \N$. 
\begin{lemma}
	\label{Lem:Characterization_parseval_frames}
	Let $\{e_n :\,n \in \N\} \subset X$ be any arbitrary orthonormal basis for $X$. The Parseval frames for $X$ are precisely the families $\{U e_n:\, n \in \N\}$, where $U: X \to X$ has an isometric adjoint $U^*$.
\end{lemma}

\begin{proof}
	Let $\{\phi_n : n \in \N\}$ by a Parseval frame for $X$. Because Parseval frames are frames, there is a surjective, bounded operator $U : X \to X$ such that $\phi_n = Ue_n$ for all $n \in \N$. Thus we have \[\norm{f}^2 = \sum_{n \in \N} \abs{\langle f,Ue_n\rangle}^2 = \sum_{n \in \N} \abs{\langle U^*f,\phi_n\rangle}^2 = \norm{U^*f}^2\,\text{for all } f \in X.\] Hence $U^*$ is an isometry.
	
	The converse follows from the same equation.
\end{proof}

Based on this observation we can characterize the operators $A: X \to X$ mapping Parseval frames to Parseval frames.

\begin{lemma}\label{Pframeb}
	Let $\cb{\phi_n: n \in\N}$ be a Parseval frame for a real or complex separable Hilbert space $X$ and let $A \in \mathcal{L}(X)$. Then, the following two statements are equivalent:
	\begin{enumerate}[label=(\roman*)]
		\item $\cb{A\phi_n: n\in\N}$ is a Parseval frame in $X$.
		\item $A^*$ is an isometry.
	\end{enumerate}
\end{lemma}
\begin{proof}
	$\cb{A\phi_n: n\in\N}$ is a Parseval frame in $X$ if and only if $\norm{f}^2 = \sum_{n \in \N} \abs{\langle f,A\phi_n\rangle}^2$ for all $f \in \Hi$. We compute
 \[\norm{f}^2 = \sum_{n \in \N} \abs{\langle f,A\phi_n\rangle}^2 = \sum_{n \in \N} \abs{\langle A^*f,\phi_n\rangle}^2 = \norm{A^*f}^2\quad \text{for all } f \in X,\] and conclude that $A^*$ is an isometry in $X$.
	
	Conversely, suppose $A^*$ is an isometry. Then \[\norm{f}^2 = \norm{A^*f}^2 = \sum_{n \in \N} \abs{\langle A^*f,\phi_n\rangle}^2 = \sum_{n \in \N} \abs{\langle f,A\phi_n\rangle}^2\quad \text{for all } f \in X.\] This implies (i). 
\end{proof}

Thus bounded linear operators $A: X \to X$ having an isometric adjoint $A^*$ are the canonical morphisms between Parseval frames. Unfortunately these operators $A$ are again no automorphisms and thus not every Parseval frame can be mapped onto every other Parseval frame. The following example will show an extreme case not allowing for creating a mapping in any direction.
\begin{example}[Incompatible Parseval frames]
	\label{ex:incomp_parseval_frames}
	Suppose $\{e_n : n \in \N\} \subset X$ is an orthonormal basis. This basis can be transformed in two Parseval frames by defining the Parseval frames 
	$$\phi_1 := \frac{1}{\sqrt{2}}e_1, \quad  \phi_2 := \frac{1}{\sqrt{2}}e_1, \quad \phi_{i+1} := e_i \quad \mbox{for all } i \geq 2,$$ 
	and 
	$$\psi_1 := e_1, \quad \psi_2 := \frac{1}{\sqrt{2}}e_2, \quad \psi_3 := \frac{1}{\sqrt{2}}e_2, \quad \psi_{i+1} := e_i \quad \mbox{for all } i \geq 3.$$
	$\{\phi_n : n \in \N\}$ is a Parseval frame for $X$ because
	\[\sum_{n \in \N} \abs{\langle f,\phi_n\rangle}^2 = 2\abs{\langle f,\frac{1}{\sqrt{2}}e_1\rangle}^2 + \sum_{n = 2}^\infty \abs{\langle f,e_n\rangle}^2 = \norm{f}^2 \quad \text{ for all } f \in X.\] 
	For a very similar reason $\{\psi_n : n \in \N\}$ is a Parseval frame for $X$. However, it is clearly impossible to define an elementwise mapping of the first Parseval frame onto the second  or vice-versa.
\end{example}

Hence the structure of the set of all Parseval frames for an Hilbert space $X$---even though it is more similar to orthonormal bases---is still too general frame-like for rebricking all pairs of Parseval frames.

\begin{example}
	\label{ex:nonrebrickable_parseval_frames}
	Let  $\{e_n : n \in \N\} \subset \Hi$ be an orthonormal basis for $\Hi$. We construct two Parseval frames by defining
	$$\phi_1 := 0, \quad  \phi_{i+1} := e_i \quad \mbox{for all }  i \geq 1,$$ 
	and 
	$$\psi_1 := e_1,\quad \psi_2 := 0,\quad \psi_{i+1} = e_i \quad \mbox{for all } i \geq 2. $$
	Again there is no mapping from one Parseval frame onto the other in any direction. 
	But their normalized sum $\{\frac{1}{\sqrt{2}}(\phi_n + i\psi_n) :\, n \in \N\}$ is a complex Parseval frame for $\Hi + i\Hi$.
\end{example}

Similar to the frame setting we keep the ansatz and focus on pairs of Parseval frames $\{ \phi_n :\, n \in \N\}, \{ \psi_n :\, n \in \N\} \subset \Hi$ such that there exists an operator $A\in \mathcal{L}(\Hi)$ having the property $\psi_n = A \phi_n$ for all $n \in \N$.

In order to get a complex Parseval frame a renormalization similar to the orthonormal basis setting is needed.  I.e. it is natural to use operators $B: = \frac 1{\sqrt 2} (\id+iA)$ where $A: \Hi \to \Hi$ such that $A^*$ is an isometry, cf. Subsection \ref{ssec Orthonormal Bases} and Theorem \ref{S A unitaer und selbstadjungiert fuer ONB}.

\begin{theorem}
	Suppose $\Hi$ is a real separable Hilbert space. Let $\cb{\phi_n: n \in\N}$ be a Parseval frame for $\Hi$ and let $A\in \mathcal{L}(\Hi)$. Then, the following two statements are equivalent:
	\begin{enumerate}[label=(\roman*)]
		\item $\cb{ B\phi_n := \frac 1{\sqrt 2} \rb{\phi_n+iA\phi_n}: n\in\N}$ is a Parseval frame in $\Hi+i\Hi$.
		\item $A$ is unitary and self-adjoint (thus $A^*$ is an isometry).
	\end{enumerate}
\end{theorem}

\begin{proof}
	(i) $\Rightarrow$ (ii)
	
	Due to Lemma \ref{Pframeb}, $\cb{\frac 1{\sqrt 2} \rb{\phi_n+iA\phi_n}: n\in\N}$ is a Parseval frame in $\Hi+i\Hi$ if and only if the operator $\frac 1{\sqrt 2} (\id+iA)^* = \frac{1}{\sqrt{2}}\left(\id - iA^*\right)$ is an isometry. This is equivalent to \[\id_{\Hi + i\Hi} = \frac{1}{\sqrt{2}}\left(\id_{\Hi + i\Hi} + iA\right) \frac{1}{\sqrt{2}}\left(\id_{\Hi + i\Hi} + iA\right)^* = \frac{1}{2}\left(\id_{\Hi + i\Hi} + AA^*\right) + \frac{i}{2}\left(A - A^*\right).\] 
	
	For  $f \in \Hi$ this implies from the real part $AA^* = \id_\Hi$, i.e. $A^*$ is an isometry, and from the imaginary part $A = A^*$, i.e. $A$ is self-adjoint. Thus $\id_\Hi = A^*A = AA^*$.
	
	(ii) $\Rightarrow$ (i)
	Conversely, if $A$ is unitary then $A$ is continuously invertible. If $A$ is self-adjoint, the spectrum is real. Thus $\frac{1}{\sqrt{2}}\left(\id_{\Hi + i\Hi} + iA\right)$ is continuously invertible and we get \[\id_{\Hi + i\Hi} = \frac{1}{\sqrt{2}}\left(\id_{\Hi + i\Hi} + iA\right) \frac{1}{\sqrt{2}}\left(\id_{\Hi + i\Hi} + iA\right)^*.\] Hence $B  = {\sqrt{2}}\left(\id_{\Hi + i\Hi} + iA\right)$ is unitary itself, and by Lemma \ref{Lem:Characterization_parseval_frames} $B^*$ is an isometry. Using Lemma \ref{Pframeb} the claim follows.
\end{proof}

\begin{remark}
	Interestingly our ansatz directly implies that the operator $A$ is unitary and self-adjoint. Thus we obtain the very same conditions which we got in the orthonormal basis case. This directly gives us a few noticeable facts.
	\begin{enumerate}[label=(\roman*)]
		\item In contrast to the general frame case the rebrickability relation is symmetric again because $A$ is invertible.
		\item To get a true Parseval frame (which is not an orthonormal basis) it is impossible to start with an orthonormal basis in the real part because the imaginary part and the rebricked Parseval frame would be an orthonormal basis again.
	\end{enumerate}
\end{remark}

\section{Summary}
As an overview, the following table shows the different types of generating systems $\{\phi_n:\ n\in\N\}$ in a Hilbert space $\Hi$ and the respective equivalent conditions for an operator $A\in \mathcal{L}(\Hi)$ such that
\begin{enumerate}
	\item[a)] $\{A\phi_n:\ n\in\N\}$ forms a frame / Parseval frame / Riesz basis / orthonormal basis,
	\item[b)] $\cb{B\phi_n := \rb{\phi_n+i  A\phi_n}:\ n\in\N}$ forms a frame / Riesz basis or\\ $\cb{B\phi_n := \frac 1{\sqrt 2}\rb{\phi_n+i  A\phi_n}:\ n\in\N}$ forms a Parseval frame / orthonormal basis.
\end{enumerate}
\vspace{1cm}
\begin{tabular}{|c|c|c|c|c|}
	\hline
	$\{\phi_n:\ n\in\N\}$ is a & Frame & Parseval frame & Riesz basis & ONB \\ \hline
	$\{A\phi_n:\ n\in\N\}$ & \multirow{2}{*}{$A$ surjective} & \multirow{2}{*}{$A^*$ isometry} & \multirow{2}{*}{$A$ invertible} & \multirow{2}{*}{$A$ unitary} \\
	is the same iff & & & & \\ \hline
	$\{B\phi_n:\ n\in\N\}$ & \multirow{2}{*}{$-i \notin \sigma_\mathrm{app}(A^*)$} & $A$ unitary \& & \multirow{2}{*}{$-i\notin\sigma(A^*)$} & $A$ unitary \& \\
	is the same iff &  & self-adjoint &  & self-adjoint \\ \hline
\end{tabular}
\vspace{0.5cm}\\
As we have observed not all rebrickable couples of frames can be obtained by this approach as can be seen in the Examples \ref{ex:nonrebrickable_frames} and \ref{ex:nonrebrickable_parseval_frames}. An approach to overcome this restriction is shown in Subsection \ref{subsec:rebricking_incomp_frames}.



 
\bibliographystyle{elsarticle-num} 
\bibliography{literature}





\end{document}